\documentclass{amsart}
\usepackage{latexsym,amsxtra,amscd,ifthen,amsmath,color,url}
\usepackage[usenames,dvipsnames]{xcolor} \usepackage{amsfonts}
\usepackage{verbatim} \usepackage{amsmath} \usepackage{amsthm}
\usepackage{amssymb}
\usepackage[linkcolor=red,citecolor=blue,urlcolor=blue]{hyperref}

\usepackage{pgf} \usepackage{tikz} 
\usetikzlibrary{shapes,shapes.geometric,arrows,fit,calc,positioning,automata,}
\usepackage[latin1]{inputenc} \usepackage{verbatim}

\numberwithin{equation}{section} \theoremstyle{plain}
\newtheorem{theorem}{Theorem}[section]
\newtheorem{lemma}[theorem]{Lemma}

\newtheorem{proposition}[theorem]{Proposition}

\newtheorem{conjecture}[theorem]{Conjecture}

 \theoremstyle{definition}
\newtheorem{definition}[theorem]{Definition}
\newtheorem{hypothesis}[theorem]{Hypothesis}
\newtheorem{example}[theorem]{Example}

\newtheorem{remark}[theorem]{Remark}

\mathchardef\mhyphen="2D

\makeatletter              
\let\c@equation\c@theorem  
\makeatother

\DeclareMathOperator{\hdet}{hdet} 
 
\DeclareMathOperator{\gldim}{gldim}

\DeclareMathOperator{\Ext}{Ext}

 \DeclareMathOperator{\GKdim}{GKdim}
\DeclareMathOperator{\End}{End}

\DeclareMathOperator{\Mod}{{\sf Mod}}

\newcommand{\op}{\textup{op}} 
 
 \newcommand{\mc}{\mathcal}

\newcommand{\kk}{\Bbbk}

\begin{document}

\title[McKay Correspondence for Hopf actions, I]{McKay
Correspondence for semisimple Hopf \\
actions on regular graded algebras, I}
\author{K. Chan, E. Kirkman, C. Walton, and J.J. Zhang}

\address{Chan: Department of Mathematics, Box 354350, University of
Washington, Seattle, WA 98195, USA}

\email{kenhchan@math.washington.edu}

\address{Kirkman: Department of Mathematics, P. O. Box 7388, Wake Forest
University, Winston-Salem, NC 27109, USA}

\email{kirkman@wfu.edu}

\address{Walton: Department of Mathematics,
The University of Illinois at Urbana-Champaign,
Urbana, IL 61801, 
USA
}
\email{notlaw@illinois.edu}

\address{Zhang: Department of Mathematics, Box 354350, University of
Washington, Seattle, WA 98195, USA}

\email{zhang@math.washington.edu}

\begin{abstract} 
In establishing a more general version of the McKay 
correspondence, we prove Auslander's theorem for actions of 
semisimple Hopf algebras $H$ on noncommutative Artin-Schelter 
regular algebras $A$ of global dimension two, where $A$ is a 
graded $H$-module algebra, and the Hopf action on $A$ is inner 
faithful with trivial homological determinant. We also show that each fixed ring $A^H$ under such an action arises as an analogue of a coordinate ring of a Kleinian singularity.
\end{abstract}

\subjclass[2010]{16T05, 16E10, 14B05}

\keywords{Artin-Schelter regular algebras, Auslander's theorem,  Hopf algebra action,
McKay correspondence, McKay quiver, trivial homological determinant}

\bibliographystyle{abbrv} \maketitle

\setcounter{section}{-1}

\section{Introduction} \label{yysec0}

Let $\kk$ be an algebraically closed field of characteristic zero. All
algebraic structures in this article are over $\kk$, and we take the
unadorned $\otimes$ to mean $\otimes_{\kk}$.

The classical McKay correspondence provides deep and elegant ties
between:
\begin{enumerate}
\item[$\bullet$]
finite subgroups of ${\text{SL}}_2(\mathbb C)$ or 
${\text{SU}}_2(\mathbb C)$ (also called {\it binary polyhedral groups});
\item[$\bullet$]
simple Lie algebras of type ADE;
\item[$\bullet$]
Platonic solids;
\item[$\bullet$]
Kleinian or DuVal singularities (that is, rational double points or quotient singularities 
${\mathbb C}^2/\widetilde{\Gamma}$ where $\widetilde{\Gamma}$ is a finite subgroup
of ${\text{SL}}_2(\mathbb C)$);
\item[$\bullet$] 
$\widetilde{\Gamma}$-Hilbert schemes;
\item[$\bullet$]
preprojective algebras of the reduced McKay quivers;
\end{enumerate}
as well as other entities in algebra, geometry and string theory. At first glance these items are unrelated, yet they are all governed by the $\mathbb{ADE}$ Dynkin diagrams; see Buchweitz \cite{Bu} for an excellent survey on the classical 
McKay correspondence, as well as McKay's original work \cite{McKay}. Moreover, a very nice review 
of the McKay correspondence and noncommutative crepant resolutions is given by Leuschke \cite{Le2}. Further, in this direction, different versions of a quantum McKay correspondence have been 
studied by many researchers, e.g. \cite{BGh, Ch, KO, Mo}. The goal of this paper 
(and the sequel \cite{PartII}) is to produce an analogue of the McKay correspondence in the context of finite dimensional Hopf actions on noncommutative regular algebras; this is applicable in both
noncommutative algebraic geometry and noncommutative invariant theory. 

\smallskip

Consider the setting below.

\begin{hypothesis}\label{yyhyp0.1}  
Unless indicated otherwise, $H$ is a 
semisimple  Hopf algebra (hence finite 
dimensional), $A$ is an {\em Artin-Schelter \textnormal{(}AS\textnormal{)} regular algebra} 
[Definition~\ref{yydef1.1}] generated in degree one, which by definition 
is connected graded. Further, $A$ is  a {\em graded $H$-module algebra} 
[Definition \ref{yydef1.3}] under an $H$-action that is 
{\em inner faithful} [Definition \ref{yydef1.5}] with {\em trivial 
homological determinant} [Definition \ref{yydef1.6}].  We also assume that $A^H \neq A$ throughout.\end{hypothesis}

When $A$ is commutative, it is a commutative polynomial 
ring, and when $H$ is cocommutative, it is the group algebra of a finite 
subgroup of $SL_n(\kk)$, with $G$ acting as graded automorphisms of $A$. 
Hence our non(co)commutative setting generalizes the setting 
of the classical McKay correspondence.

\smallskip

This paper can be viewed as a continuation of the project started in 
\cite{CKWZ}, where the finite dimensional Hopf actions on AS regular 
algebras of dimension 2 with trivial homological determinant were classified.  
The AS regular algebras of dimension 2 are isomorphic to one of the two 
families of algebras given in Example~\ref{yyex1.2}.  The classification 
of such semisimple Hopf actions on AS regular algebras of 
dimension~2 is repeated in 
Theorem \ref{yythm0.4} below.

\smallskip

A key ingredient in the classical McKay correspondence is Auslander's
Theorem, which states that if $G$ is a small finite group acting on a
commutative polynomial ring $S$, then $S\# \Bbbk G$ is isomorphic to
$\End_{S^G}(S)$ as algebras
\cite[Proposition~3.4]{Aus62} (see also \cite[Theorem~5.12]{LW}). We conjecture that a version of this result
holds for semisimple Hopf actions on noetherian AS regular algebras. 

\begin{conjecture} 
\label{yycon0.2} 
Let $H$ be a semisimple Hopf algebra and suppose that $A$ is a noetherian AS regular $H$-module algebra satisfying Hypothesis~\ref{yyhyp0.1}. Then there is a natural 
graded algebra isomorphism 
\begin{equation}
\label{E0.2.1}\tag{E0.2.1}
A\#H \cong \End (A_{A^H}).
\end{equation}
\end{conjecture}

Our main result is to show that Conjecture~\ref{yycon0.2} is true 
in dimension~2, which answers \cite[Question 8.5]{CKWZ}. 

\begin{theorem}[Theorem~\ref{yythm4.1}] 
\label{yythm0.3}
If $A$ is a noetherian AS regular algebra of dimension~$2$, then  
Conjecture~{\rm{\ref{yycon0.2}}} is true.
\end{theorem}

In \cite{CKWZ} we gave the classification of 
pairs $(A,H)$, where $H$ is a finite dimensional Hopf algebra, and $A$ 
is an AS regular algebra of global dimension $2$, generated in degree 
one, that admits a graded, inner faithful $H$-action with trivial homological determinant. We call such a 
Hopf algebra $H$ a {\it quantum binary polyhedral group}. The pairs $(A,H)$, 
where $H$ is semisimple, are described in the following theorem, on which 
the proof of Theorem \ref{yythm0.3} is based.

\begin{theorem} \textnormal{(\cite[Theorem~0.4]{CKWZ}).}
\label{yythm0.4} 
Suppose that $(A,H)$ is a pair satisfying Hypothesis~\ref{yyhyp0.1}  with $\gldim A =2$.  Then $(A,H)$ is one of the following. 

\[ \begin{array}{|c|l|l|} \hline \rm{Case} & A & H\\ 
\hline \hline 
{\tt{(a)}} & \kk[u,v] & \kk \widetilde{\Gamma}~{\rm for}~\textcolor{white}{\hat {\textcolor{black}{\widetilde{\Gamma}}}}
{\rm ~a~finite~subgroup~of~} \mathrm{SL}_2(\kk) \\ \hline 
{\tt{(b)}} &\kk_{-1}[u,v] & \kk C_n~{\rm for}~n\geq 2 \hfill {\rm (diagonal~action)}  \\ 
{\tt{(c)}} &\kk_{-1}[u,v]& \kk C_2 \hfill  {\rm (nondiagonal~action)}\\ 
{\tt{(d)}} &\kk_{-1}[u,v]& \kk D_{2n} ~{\rm for}~n\geq 3 \\ 
{\tt{(e)}} &\kk_{-1}[u,v]& 
(\kk D_{2n})^{\circ}~{\rm for}~n\geq 3  \\ 
{\tt{(f)}} &\kk_{-1}[u,v]&\mc{D}(\widetilde{\Gamma})^{\circ} 
~{\rm for}~\widetilde{\Gamma}{\rm ~a~finite~nonabel.~subgroup~of~} \mathrm{SL}_2(\kk) \\
\hline 
{\tt{(g)}} &\kk_q[u,v]~{\rm with}~ q^2 \neq 1 & \kk C_n~{\rm for}~n
\geq 2  \hfill  {\rm (diagonal~action)}\\ \hline 
{\tt{(h)}} &\kk_J[u,v] & \kk
C_2  \hfill {\rm (diagonal~action)} \\ \hline \end{array} \]

\medskip

\begin{center}
\small{\textsc{Table~1.} {\rm Quantum binary polyhedral groups and their module algebras}}
\end{center}

\bigskip

\noindent
Here $C_n$ is a cyclic group of order $n$, $D_{2n}$ is
a dihedral group of order $2n$, and $\mathcal{D}(\widetilde{\Gamma})$ 
is a cocycle deformation 
of $(\kk\widetilde{\Gamma})^\circ$ \cite{Ma}. 
The $H$-action on $A$ is defined by 
the $H$-module {\rm{(}}or equivalently, the $H^{\circ}$-comodule{\rm{)}} 
structure of $A_1 =\kk u \oplus \kk v$, as~given~below.
\begin{itemize}
{\setlength\itemindent{6pt} \item[({\tt a,d})]
$A_1$ is a faithful simple $2$-dimensional $G$-module
with $G$= $\widetilde{\Gamma}$ or $D_{2n}$, resp.}
\smallskip

{\setlength\itemindent{6pt} \item[({\tt b,g,h})] 
$\kk u$ is a faithful $C_{n}$-module and $\kk v$ is its dual.}
\smallskip

{\setlength\itemindent{6pt} \item[({\tt c})] 
$\kk (u+v)$ is the trivial $C_{2}$-module, and $\kk (u-v)$ is the sign 
representation.}
\smallskip

{\setlength\itemindent{6pt}\item[({\tt e})] 
Let $v_1 = u+\sqrt{-1} v$ and $v_2=\sqrt{-1}u+v$. Then $\kk v_1$ 
{\rm{(}}resp. $\kk v_2${\rm{)}} is 
\linebreak \textcolor{white}{.}~$1$-dimensional $\kk D_{2n}$-comodule with 
structure map $\rho (v_1) = v_1 \otimes g$ {\rm{(}}resp.
\linebreak \textcolor{white}{.}~$\rho (v_2) = v_2 
\otimes h${\rm{)}}, where $D_{2n} = \langle g,h~|~ g^{2} =h^{2} =(gh)^{n} =1 \rangle$.}
\smallskip

{\setlength\itemindent{6pt} \item[({\tt f})] 
$\mc{D} (\widetilde{\Gamma})$ is a finite dimensional Hopf 
quotient of the quantum special linear 
\linebreak \textcolor{white}{.}~group $\mathcal{O}_{-1} (SL_{2} ( \kk ))$ generated by 
$(e_{i,j} )_{1 \leq i,j \leq 2}$, which coacts on $A_1$ from 
\linebreak \textcolor{white}{.}~the right via
$\rho (u) =u \otimes e_{1,1} +v \otimes e_{2,1}$ and $\rho (v) =u \otimes
e_{1,2} +v \otimes e_{2,2}$.}
\end{itemize}

\vspace{-.15in} \qed
\end{theorem} 

We also have the following result for the fixed subrings $A^H$. 
These rings are referred to as {\it quantum
Kleinian singularities} [Definition~\ref{yydef5.1}].

\begin{theorem} \textnormal{(\cite[Theorem~0.6]{CKWZ} and Theorem~\ref{yythm5.2}).}
\label{yythm0.5}
Suppose that $(A,H)$ is a pair satisfying Hypothesis~\ref{yyhyp0.1} with $\gldim A =2$. Then $A^H$ is not AS regular, and is isomorphic to $C/C\Omega$, for $C$ an AS regular
algebra of dimension $3$ and $\Omega$ a normal element in $C$.
\end{theorem}

We also define in a natural way, the {\it McKay quiver} of a semisimple Hopf 
action on a connected algebra $A$ in Definition~\ref{yydef2.3}. In the sequel to this work \cite{PartII}, we obtain a correspondence between this quiver and the {\it Gabriel quiver} defined in terms of projective $A\#H$-modules. Various correspondences between module categories over $A^H$, over $A\#H$, and over $\End(A_{A^H})$ are presented, as well, in \cite{PartII}.

\medskip

The paper is organized as follows.  Section~1 contains definitions 
and preliminary results. In Section~2, we define the McKay quiver as discussed above, and we recall 
a result from \cite{CKWZ} concerning the McKay quivers that can occur under our 
standing  hypotheses with $\gldim A =2$. Section~3 contains results on the fixed subrings arising from actions with McKay quiver of types $\mathbb{D}$ and $\mathbb{E}$; this is one of key steps 
in the proof of Theorem~\ref{yythm0.3}. Section 4 contains the complete
proof of our version of Auslander's Theorem  for semisimple Hopf actions on AS regular 
algebras of dimension 2 [Theorem~\ref{yythm0.3}] by considering the possible pairs $(A,H)$ 
that satisfy our standing hypotheses (from Theorem~\ref{yythm0.4}). In Section 5, we prove Theorem 
\ref{yythm0.5}.


\section{Preliminaries} 
\label{yysec1}

In this section, we provide background material on Artin-Schelter
regular algebras, Hopf algebra actions on graded algebras, 
and the homological determinant of such actions. 
Note that we usually work 
with left modules, and a right $A$-module is identified with 
a (left) $A^{\op}$-module, where $A^{\op}$ is the opposite ring of $A$. Let $A\mhyphen \Mod$ denote the category of left modules over an algebra $A$.
Moreover, for every ${\mathbb Z}$-graded algebra $A=\bigoplus_{i\in {\mathbb Z}}A_i$,
and for each positive integer $d$, the {\it $d$-th Veronese subring} of $A$ 
is defined to be
$A^{(d)}:=\bigoplus_{i\in {\mathbb Z}}A_{id}.$

\subsection{Artin-Schelter regular algebras} 
\label{yysec1.1}

An algebra $A$ is said to be {\it connected graded} if $A = \bigoplus_{i \in \mathbb{N}} A_i$ with $A_0 = \kk$ and  $A_i \cdot A_j \subseteq
A_{i+j}$ for all $i,j \in \mathbb{N}$, and a graded algebra is called {\it locally finite} if  $\dim_\kk A_i < \infty$ for all $i$. In this case, the {\it Hilbert series} of $A$
is defined to be 
$H_A(t):=\sum_{i \in \mathbb{N}} (\dim_\kk A_i) t^i.$ 

\smallskip

Here we consider a class of noncommutative
graded algebras that serve as noncommutative analogues of commutative
polynomial rings. These algebras are defined as follows.

\begin{definition} 
\label{yydef1.1} 
A connected graded algebra $A$ is called \emph{Artin-Schelter (AS) Gorenstein} if the following conditions hold:
\begin{enumerate} 
\item[(i)]
$A$ has finite injective dimension $d<\infty$ on both sides, 
\item[(ii)]
$\Ext^i_A(\kk, A) = \Ext^i_{A^{\op}}(\kk, A)= 0$ for
all $i \neq d$ where $\kk = A/A_{\geq 1}$, and 
\item[(iii)]
$\Ext^d_A(\kk, A)\cong \kk(\ell)$ and $\Ext^d_{A^{\op}} (\kk, A) 
\cong  \kk(\ell)$ for some integer $\ell$. 
\end{enumerate} 
The integer $\ell$ is called the \emph{AS index} of $A$. If moreover, 
\begin{enumerate} 
\item[(iv)] 
$A$ has finite global dimension $d$, 
\item[(v)] 
$A$ has finite Gelfand-Kirillov dimension, 
\end{enumerate} 
then $A$ is called \emph{Artin-Schelter (AS) regular} of dimension $d$. 
\end{definition}


\begin{example} 
\label{yyex1.2} The AS regular algebras of global
dimension~2, that are generated in degree one, are listed below 
(up to isomorphism): 
\begin{enumerate} 
\item[(i)]
the {\it Jordan plane}: $\kk_J[u,v] := 
\kk \langle u,v \rangle/ (vu - uv - u^2)$, and 
\item[(ii)]
the {\it skew polynomial ring}: $\kk_q[u,v] := 
\kk \langle u,v \rangle/(vu - quv)$ for $q \in \kk^{\times}$. 
\end{enumerate} 
It is well-known that these algebras are noetherian domains. 
\end{example}

\subsection{Hopf algebra actions} 
\label{yysec1.2}
Throughout $H$ stands for a Hopf algebra over $\kk$ 
with structural notation $(H, m, u, \Delta, \epsilon, S)$, and we write $\Delta(h) =\sum h_1\otimes h_2$ (Sweedler notation). We
recommend \cite{Mo1} as a basic reference for the theory of Hopf
algebras.

\begin{definition} 
\label{yydef1.3}
Let $A$ be an algebra and $H$ a Hopf algebra. 
\begin{enumerate} 
\item[(1)] 
We say that $A$ is a  {\it {\rm{(}}left{\rm{)}} $H$-module algebra}, or that $H$ 
{\it acts on} $A$, if $A$ is an algebra in the category of left 
$H$-modules. Equivalently, $A$ is a left $H$-module such that
\begin{eqnarray*}
h(ab)=\sum h_1(a)
h_2(b) &\mathrm{and}& h(1_A) = \epsilon(h)1_A
\end{eqnarray*}
 for all $h\in
H$, and all $a, b\in A$. 

\smallskip

\item [(2)]
We say that $A$ is a {\it graded $H$-module algebra} if $A$ is an
algebra in the category of ${\mathbb N}$-graded (left) $H$-modules 
(with elements in $H$ having degree zero), or equivalently, 
each homogeneous component of $A$ is a left $H$-submodule.

\smallskip

\item[(3)] 
Given a $H$-module algebra $A$, the {\it smash product algebra $A\# H$} 
is  $ A\otimes H$ as a $\kk$-vector space, with 
multiplication defined by
\begin{eqnarray*}
(a\# h)(a'\# h') &=& \sum a h_1(a')\# h_2h'
\end{eqnarray*}
for all $h, h' \in  H$ and $a, a' \in  A$. 
\end{enumerate} 
\end{definition}

We identify $H$ (resp. $A$) with a subalgebra of $A\# H$ via the map $i_H :
h\mapsto 1\#h$ for all $h\in H$ (resp. the map $i_A : a \mapsto a\#1$ for all $a \in A$).

\smallskip

If $A$ is a graded $H$-module algebra, then $A\# H$
is graded with $\deg h=0$ for all $0\neq h\in H$ and $A$ is a graded
subalgebra of $A\# H$. If, further, $A$ is locally finite 
and $H$ is finite dimensional, then $A\# H$ is locally finite and 
$(A\# H)_i=A_i\# H$ for all $i$. 

\begin{remark}
\label{yyrem1.4} 
One can form a right-handed version of the smash product
algebra, denoted by $H\# A$, for a right $H$-module algebra $A$. When
the antipode $S$ is bijective as assumed, these two versions can be 
exchanged  \cite[Lemma 2.1]{RRZ}. We work with the left-handed 
version of smash product and would like to remark that all 
statements can be transformed from the left-handed version to the 
right-handed one. 
\end{remark}

We want to restrict ourselves to Hopf ($H$-)actions that do not factor
through the action of a proper Hopf quotient of $H$.

\begin{definition} 
\label{yydef1.5} 
Let $M$ be a left $H$-module. We say that $M$ is an {\it inner faithful} 
$H$-module, or $H$ \emph{acts inner faithfully} on $M$, if $IM\neq 0$ 
for every nonzero Hopf ideal $I$ of $H$. The same terminology applies 
to $H$-module algebras $A$.
\end{definition}

We now recall the homological determinant of a Hopf algebra action on an
Artin-Schelter regular algebra below. Recall that by \cite[Corollary D]{LPWZ}, a
connected graded algebra $A$ is AS regular if and only if the $\Ext$-algebra 
$E=\bigoplus_{i\geq 0} \Ext^i_A(_A\kk,_A\kk)$ of $A$ is Frobenius. 

\begin{definition} \cite{KKZ2}
\label{yydef1.6} 
Retain the notation above. Let $A$ be an AS regular algebra with 
Frobenius $\Ext$-algebra $E$. Suppose ${\mathfrak e}$ is a 
nonzero element in $\Ext^d_A(_A\kk,_A\kk)$ where $d=\gldim A$.
Let $H$ be a Hopf algebra acting on $A$ from
the left. By \cite[Lemma 5.9]{KKZ2}, $H$ acts on $E$ from the left. 
\begin{enumerate}
\item[(1)] The
{\it homological determinant} of the $H$-action on $A$, denoted by $\hdet_H A$, is defined to be $\eta\circ S$ where $\eta: H\to k$ is
determined by 
$h\cdot {\mathfrak e} = \eta(h) {\mathfrak e}$.
\item[(2)] The
homological determinant is {\it trivial} if $\hdet_H A=\epsilon$.
\end{enumerate}
\end{definition}

In the case that $A$ is a commutative polynomial ring $\kk[x_1, \dots,
x_n]$, and $H = kG$, for $G$ a finite subgroup of $GL_n(k)$, $\hdet_HA$
becomes the ordinary determinant $\det_G : \kk G \rightarrow \kk$
\cite[Remark 3.4(c)]{KKZ2}. In this case, the homological determinant of the
$\kk G$-action on $A$ is trivial if and only if $G \leq
SL_n(\kk)$.

\section{McKay quivers} 
\label{yysec2} 

One of the most important objects in the McKay correspondence is the McKay 
quiver. In this section we present a version of such a diagram in the semisimple Hopf action
setting.  First, we begin with a discussion of fusion rings.

\smallskip

Given a Hopf algebra $H$ over $\kk$, we define a ring associated to a
set of nonisomorphic finite dimensional $H$-modules. 

\begin{definition} 
\label{yydef2.1}
Let $H$ be a Hopf algebra. Let $\{ [ V_{i} ] \}_{i \in J}$ be the set of all 
isomorphism classes of finite dimensional left $H$-modules, indexed by some
set $J$. We define the {\it Grothendieck ring $G_0(H):=G_0(_HH)$ of $H$} to be an 
associative ${\mathbb Z}$-algebra generated by $\{[V_i]\}_{i\in J}$, 
subject to relations $[V]=[U]+[W]$ for each short exact sequence of 
$H$-modules 
$$0\to U \to V\to W\to 0.$$ 
Then $G_0(H)$ has a ${\mathbb Z}$-basis $\{ [ V_{i} ] \}_{i \in I}$
consisting of isomorphism classes of all  finite dimensional simple left 
$H$-modules where $I$ is a subset of $J$. 
Moreover, for $i,j\in I$, the multiplication in $G_0(H)$ is 
determined by 
\begin{equation}
\label{E2.1.1}\tag{E2.1.1} 
[V_i]\cdot [V_j] = \sum_{k\in I} n_{i,j}^k [V_k],
\end{equation} 
where $n_{i,j}^k\in {\mathbb N}$ is the multiplicity of the finite
dimensional simple left $H$-module $V_k$
in the composition series of the $H$-module $V_i\otimes V_j$. The unit 
of $G_0(H)$ is $[\kk]$. 
\end{definition}

We arrive at our definition of a fusion ring.

\begin{definition} 
\label{yydef2.2} Let $F$ be an associative ring over
$\mathbb{Z}$, with multiplicative identity, equipped with an
augmentation map $\pi: F \rightarrow \mathbb{Z}$. We call $F$ a {\it fusion ring} if there is a
${\mathbb Z}$-module basis $\{X_i\}_{i\in I}$ of $F$, for some index set $I$, such that
\begin{enumerate} 
\item 
$1_F\in \{X_i\}_{i\in I}$; 
\item 
$\pi(X_i)\geq 1$, for all $i\in I$; and 
\item 
for all $i,j\in I$, we have $X_i 
X_j=\sum_{k\in I} n_{i,j}^k X_k$ for some $n_{i,j}^k \in {\mathbb N}$.
\end{enumerate} 
The coefficients $n_{i,j}^k$ are called the {\it fusion rule
coefficients}. 
\end{definition}

The Grothendieck ring $G_0(H)$ is an example of a fusion ring, where 
$X_i = [V_i]$ and $\pi([V_i]) = \dim_{\kk} V_i$ for $i\in I$.

\medskip

Now we present the definition of the McKay quiver 
for a semisimple Hopf algebra action on a connected graded 
algebra. This is achieved by defining the McKay quiver for an 
{\it effective} element $Y$ in a fusion ring $F$, where $Y$
is a nonzero element in $F$ so that $Y=\sum_{i\in I} \alpha_i X_i$
for some $\alpha_i\geq 0$.

\begin{definition} 
\label{yydef2.3} 
Let $F$ be a fusion ring with a 
${\mathbb Z}$-basis $\{X_i\}_{i\in I}$, and  let $Y\in F$ be an effective 
element. 
\begin{enumerate} 
\item[(1)] 
The \emph{right McKay quiver} ${\mathcal M}(Y)$ is defined as follows: 
\begin{enumerate} 
\item
the vertices of ${\mathcal M}(Y)$ are $X_i$, where $i\in I$, 
\item 
there are exactly $m_{ij}$ arrows from $X_i$ to $X_j$ in ${\mathcal M}(Y)$,
where $m_{ij}$  is determined by $X_j Y=\sum_{i\in I} m_{ij} X_i$
\end{enumerate} 
The \emph{left McKay quiver} $(Y){\mathcal M}$ is
defined similarly, where $m_{ij}$ is determined by $YX_j=\sum_{i\in I}
m_{ij} X_i$. 

\smallskip

\item[(2)] 
If $m_{ij}=m_{ji}$ for all $i,j\in I$, then the \emph{right 
Euclidean diagram}, 
still denoted by ${\mathcal M}(Y)$, is an undirected graph defined 
as follows:
\begin{enumerate}
\item
the vertices of ${\mathcal M}(Y)$ are $X_i$ where $i\in I$,
\item
the multiplicity of the edge connecting $X_i$ to $X_j$  is $m_{ij}$.
\end{enumerate}
The \emph{left Euclidean diagram} $(Y){\mathcal M}$ is defined 
similarly.

\smallskip

\item[(3)]
The connected component containing $1_F$ is called the \emph{principal 
component} of ${\mathcal M}(Y)$ (resp. $(Y){\mathcal M}$) and 
is denoted by ${\mathcal M}_0(Y)$ (resp. $(Y)_0{\mathcal M}$).

\smallskip

\item[(4)] 
The positive number $\pi(Y)$ [Definition \ref{yydef2.2}] is called the 
\emph{rank} of the McKay quiver ${\mathcal M}(Y)$, and of the 
principal component ${\mathcal M}_0(Y)$.
\end{enumerate} 
\end{definition}

If $F$ is noncommutative, then the left McKay quiver could be different
from the right one, although these two versions are isomorphic in many
natural examples. On the other hand, if $F$ is commutative, then there
is no difference between left and right McKay quivers.

\smallskip

The McKay quivers for the actions of the quantum binary polyhedral
groups on AS regular algebras of dimension 2 are discussed briefly 
in \cite[Section 7]{CKWZ}. We repeat that result below.

\begin{proposition}
\cite[Proposition~7.1]{CKWZ} 
\label{yypro2.4}
Let $H$ be a semisimple quantum binary polyhedral group that 
acts on an AS regular algebra $A$ of dimension 2 generated in 
degree one as in Theorem~{\rm{\ref{yythm0.4}}}. Take the 
effective element $Y$ to be the degree one graded piece $A_1$ 
of $A$ in the fusion ring $G_0(H)$. Then the left Euclidean diagrams
{\rm{(}}and hence the left McKay quivers, which have an arrow in both directions for each edge in the diagram {\rm{)}} 
$\mc{M}(Y)$ for these actions are given in the  table below. \qed

{\small 
$$\begin{array}{|c|c|c|c|c|c|c|c|c|} \hline {\rm Case} &~~
{\tt{(a)}}~~ &~~{\tt{(b)}}~~&~~{\tt{(c)}}~~&~~{\tt{(d)}}~~
&~~{\tt{(e)}}~~&~~{\tt{(f)}}~~&~~{\tt{(g)}}~~&~~{\tt{(h)}}~~\\ \hline
&&&&&&&&\\ \text{\begin{tabular}{c}{\rm Left}\\{\rm Euclidean}\\ {\rm diagram}\end{tabular}}
&\widetilde{\mathbb{A}}$-$\widetilde{\mathbb{D}}$-$\widetilde{\mathbb{E}}&\widetilde{\mathbb{A}}_{n-1}&\widetilde{\mathbb{L}}_1& 
\begin{cases} \widetilde{\mathbb{D}}_{\frac{n+4}{2}}, &
\hspace{-.1in} n~{\rm even}\\ \widetilde{\mathbb{DL}}_{\frac{n+1}{2}},
&\hspace{-.07in} n~{\rm odd} \end{cases}
&\widetilde{\mathbb{A}}_{2n-1}&\widetilde{\mathbb{D}}$-$\widetilde{\mathbb{E}}&\widetilde{\mathbb{A}}_{n-1}&
\widetilde{\mathbb{A}}_1\\ &&&&&&&&\\ \hline 
\end{array}
$$
}
\medskip
\begin{center}
\small{\textsc{Table~2.} {\rm Left McKay quivers for actions of quantum
binary polyhedral groups}}
\end{center}
\end{proposition}

\noindent 
We refer the reader to \cite{HPR} for illustrations of the
Euclidean diagrams listed above.

\section{Fixed subrings}
\label{yysec3}

The proof of Theorem \ref{yythm0.3} is delicate, and, among other 
things, requires a deep understanding of the fixed subrings 
in types $\mathbb{D}$ and $\mathbb{E}$. A direct calculation of the fixed 
subrings in types $\mathbb{D}$ and $\mathbb{E}$ is too complicated (or even impossible) 
in the noncommutative case, so our approach is to use the McKay quivers to 
convert the noncommutative case to the commutative case.

\subsection{Fixed subrings under Hopf actions of type $\mathbb{E}$}
\label{yysec3.1}  In this subsection we study type $\mathbb{E}$, In type $\mathbb{E}$ we have $A = \kk[u,v]$ 
or $\kk_{-1}[u,v]$ by Theorem~\ref{yythm0.4} and Proposition~\ref{yypro2.4}.
We will show that for Hopf algebras $H$ of type 
$\mathbb{E}$ in Theorem \ref{yythm0.4} we have that $\kk_{-1}[u,v]^H$ is 
isomorphic to $\kk[u,v]^{\widetilde{\Gamma}}$ for a corresponding 
subgroup $\widetilde{\Gamma}$ of SL$_2(\kk)$. Recall that $\widetilde{\Gamma}$ arises as a nonsplit central extension of $\mathbb{Z}_2$ by the corresponding polyhedral group~$\Gamma$: $0 \rightarrow \mathbb{Z}_2 \rightarrow \widetilde{\Gamma} \rightarrow \Gamma \rightarrow 0$.
\smallskip

By Theorem \ref{yythm0.4}, for each given 
diagram $\widetilde{\mathbb E}_i$, there are at most 
three Hopf algebras $H$ associated to it (see cases ({\tt a},{\tt f})
in Theorem \ref{yythm0.4}), and each Hopf algebra $H$ satisfies
\begin{equation}
\label{E3.0.1}\tag{E3.0.1}
\kk \to \kk {\mathbb Z}_2 \to H\to \kk \Gamma \to \kk
\end{equation}
where $\Gamma$ is the (unique) corresponding polyhedral group
of type $\widetilde{\mathbb E}_i$. 

\medskip

\noindent {\it Notation.} Let $\{M(d,j)\}_{j=1}^{s_d}$ be the complete set of isomorphism classes 
of simple left $H$-modules of dimension $d$, for some  $s_d\geq 0$.
Let 
$\kk \langle M(d,j)\rangle$ (resp. $\kk [M(d,j)]$) be the free 
algebra (resp. the commutative polynomial ring) generated by 
$M(d,j)$. 

\medskip

 For 
example, considering the diagram $\widetilde{\mathbb E}_6$ as in
 \cite[p.5]{HPR}, there are seven simple left $H$-modules
$${\text{$M(1,1), ~M(1,2), ~M(1,3), ~M(2,1), ~M(2,2), ~M(2,3), ~M(3,1)$.}}$$
We choose $M(1,1)$  to be the trivial module, and $M(2,1)$ to be 
the 2-dimensional simple module that is connected with $M(1,1)$. 
Using the notation in Section \ref{yysec2}, $Y=M(2,1)=\kk u\oplus 
\kk v$, which is the generating space of the AS regular algebra
$A$ of dimension two. Let $M(3,1)$ be the unique 3-dimensional simple
module that is connected with $M(2,1)$ in the diagram 
$\widetilde{\mathbb E}_6$.\\ 

\bigskip

\hspace{.4in} $\widetilde{\mathbb E}_6$:

\vspace{-.2in}

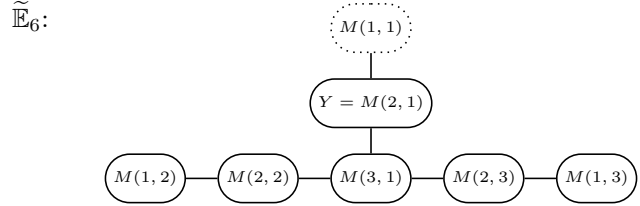
\begin{center}
\tikzset{elliptic state/.style={draw,ellipse}}
{\tiny 
\begin{tikzpicture}
[node distance=2cm, semithick,auto]
  \tikzstyle{every state}=[fill=white,draw=black, rounded rectangle, text=black]
  \node[state] at (0,0) (A)        {$M(3,1)$};
  \node[state] at (1.5,0) (B)        {$M(2,3)$};
  \node[state] at (-1.5,0) (C)        {$M(2,2)$};
  \node[state] at (0,1) (D)       {$Y=M(2,1)$};
  \node[state] at (3,0) (E)      {$M(1,3)$};
  \node[state] at (-3,0) (F)      {$M(1,2)$};
  \node[state,dotted] at (0,2) (G)   {$M(1,1)$};

  \path (A) edge              node {} (B)
            edge              node {} (C)
            edge              node {} (D)
        (B) edge              node {} (E)
        (C) edge              node {} (F)
        (D) edge              node {} (G);
\end{tikzpicture}
}
\end{center}

\bigskip

We have the following results.

\begin{lemma}
\label{yylem3.1} 
Suppose the McKay quiver corresponding to $(A,H)$ is $\widetilde{\mathbb E}_6$. Then the following statements hold.
\begin{enumerate}
\item[(1)]
$\kk\langle M(3,1)\rangle_2=
M(3,1)\otimes M(3,1)=M(3,1)^{\oplus 2}\oplus M(1,1)\oplus M(1,2)\oplus M(1,3)$ as $H$-modules.

\smallskip

\item[(2)]
$\kk[M(3,1)]_2=M(3,1)\oplus M(1,1)\oplus M(1,2)\oplus M(1,3)$ as $\Gamma$-modules.

\smallskip

\item[(3)]
There is a unique 1-dimensional subspace $\kk r$ of $\kk[M(3,1)]_2$, such that $A^{(2)}$ is 
isomorphic to $\kk[M(3,1)]/(r)$ as graded $\Gamma$-module algebras.

\smallskip

\item[(4)]
For each $H$ of type $\widetilde{\mathbb E}_6$, $A^{(2)}$
is isomorphic to $\kk[u,v]^{(2)}$ as graded $\Gamma$-module algebras.
As a consequence, $ A^H\; \cong \kk[u,v]^{\widetilde{\Gamma}}$ as graded 
algebras.
\end{enumerate}
\end{lemma}

\begin{proof} (1) Note that $Y \otimes Y= M(3,1) \oplus M(1,1)$ according to 
the diagram $\widetilde{\mathbb E}_6$. Further, by the diagram $\widetilde{\mathbb E}_6$, we have
$$\begin{aligned}
\; [M(3,1)\otimes M(3,1)] \oplus M(3,1) &=  M(3,1)\otimes [M(3,1)\oplus M(1,1)]\\
&= M(3,1)\otimes [Y\otimes Y]=[M(3,1)\otimes Y]\otimes Y\\
&=[M(2,1)\oplus M(2,2)\oplus M(2,3)] \otimes Y\\
&=M(3,1)^{\oplus 3} \oplus M(1,1)\oplus M(1,2)\oplus M(1,3).
\end{aligned}
$$
The assertion follows from the above equation by cancelling a copy of $M(3,1)$.
\smallskip

(2) Note that $\kk[M(3,1)]_2$ is a quotient of $\kk\langle M(3,1)\rangle_2$.
Moreover, $Y \otimes Y = M(3,1) \oplus M(1,1)$ implies that $M(3,1) = A_{2}$ 
is the generating space of the Veronese subring $A^{(2)}$. Now $A = \kk[u,v]$ 
or $\kk_{-1}[u,v]$ by Theorem~\ref{yythm0.4} and Proposition~\ref{yypro2.4}, 
and $A^{(2)}$ is respectively $\kk[a,b,c]/(ab-c^2)$ or $\kk[a,b,c]/(ab+c^2)$ 
for $a =u^2$, $b=v^2$, $c=uv$. So $A^{(2)} $ is isomorphic as 
$\Gamma$-module algebras to $\kk [M(3,1)]/(r)$ for some relation $r$ in 
degree two. This implies that $\kk[M(3,1)]_2$ has a quotient $\Gamma$-module of 
dimension 5. The assertion follows by a $\Bbbk$-vector space dimension argument. 
\smallskip

(3) As a commutative algebra,  $A^{(2)}$ is isomorphic to $\kk [M(3,1)]/(r)$  
for some $r$ in degree two by the arguments in (2). Since $A^{(2)}$ is a $\Gamma$-module algebra,
$\kk r$ is a left $\Gamma$-module. Hence, $\kk r$ is isomorphic to either
$M(1,1)$, or $M(1,2)$ or $M(1,3)$. Note that 1-dimensional $\Gamma$-modules 
$M(1,1)$, $M(1,2)$ and $M(1,3)$ are all non-isomorphic. The choice of $r$ 
is uniquely determined by using the result \cite[Proposition 2.4]{JiZ2} and 
the~fact that the $\Gamma$-action on $A^{(2)}$ has trivial homological determinant 
\cite[Lemma~2.6(c)]{CKWZ}.
\smallskip

(4) This follows from (3) and \eqref{E3.0.1}; see also \cite[Lemma~6.27(b)]{CKWZ}.
\end{proof}

The next two lemmas are similar to the previous one. We prove the 
first one and provide the corresponding diagrams for both results.

\vspace{.4in}

\hspace{.4in} $\widetilde{\mathbb E}_7$:

\vspace{-.2in}

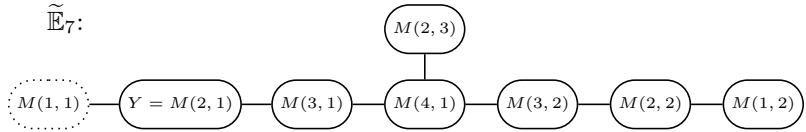
\begin{center}
\tikzset{elliptic state/.style={draw,ellipse}}
{\tiny 
\begin{tikzpicture}
[node distance=2cm, semithick,auto]
  \tikzstyle{every state}=[fill=white,draw=black,rounded rectangle,text=black]
  \node[state] at (0,0) (A)        {$M(4,1)$};
  \node[state] at (1.5,0) (B)        {$M(3,2)$};
  \node[state] at (-1.5,0) (C)        {$M(3,1)$};
  \node[state] at (0,1) (D)       {$M(2,3)$};
  \node[state] at (3,0) (E)      {$M(2,2)$};
  \node[state] at (-3.25,0) (F)      {$Y=M(2,1)$};
  \node[state,dotted] at (-5,0) (G)   {$M(1,1)$};
   \node[state] at (4.5,0) (H)      {$M(1,2)$};

  \path (A) edge              node {} (B)
            edge              node {} (C)
            edge              node {} (D)
        (B) edge              node {} (E)
        (C) edge              node {} (F)
        (E) edge              node {} (H)
        (F) edge              node {} (G);
\end{tikzpicture}
}
\end{center}

\bigskip

\begin{lemma}
\label{yylem3.2} 
Suppose the McKay quiver corresponding to $(A,H)$ is $\widetilde{\mathbb E}_7$. Then the following statements hold.

\begin{enumerate}
\item[(1)]
$\kk\langle M(3,1)\rangle_2=
M(3,1)\otimes M(3,1)=M(3,1)\oplus M(3,2)\oplus M(2,3)\oplus M(1,1)$ as $H$-modules. 

\smallskip

\item[(2)]
$\kk[M(3,1)]_2=M(3,i)\oplus M(2,3)\oplus M(1,1)$ where $i$ =1 or 2 as $\Gamma$-modules.

\smallskip

\item[(3)]
There is a unique 1-dimensional subspace $\kk r$ of $\kk[M(3,1)]_2$, such that $A^{(2)}$  is 
isomorphic to $\kk[M(3,1)]/(r)$ as graded $\Gamma$-module algebras.
In fact, $\kk r$ is the submodule $M(1,1)$ in part {\rm{(2)}}. 

\smallskip

\item[(4)]
For each $H$ of type $\widetilde{\mathbb E}_7$, $ A^{(2)}$
is isomorphic to $\kk[u,v]^{(2)}$ as graded $\Gamma$-module algebras.
As a consequence, $A^H\; \cong \kk[u,v]^{\widetilde{\Gamma}}$ as graded 
algebras.
\end{enumerate}
\end{lemma}

\begin{proof} (1) By the diagram $\widetilde{\mathbb E}_7$, we have
$$\begin{aligned}
\; [M(3,1)\otimes & M(3,1)] \oplus M(3,1) =  M(3,1)\otimes [M(3,1)\oplus M(1,1)] \\
& =[M(3,1)\otimes Y]\otimes Y\\
&=[M(2,1)\oplus M(4,1)] \otimes Y\\
&=M(1,1)\oplus M(3,1) \oplus M(3,1)\oplus M(3,2)\oplus M(2,3).
\end{aligned}
$$
The assertion follows from the above equation by cancelling a copy of $M(3,1)$.

\smallskip

(2) Note that $(\kk[M(3,1)])_2$ is a quotient of $(\kk\langle M(3,1)\rangle)_2$
and that $(\kk[M(3,1)])_2$ has a quotient $\Gamma$-module of dimension 5. Hence 
the assertion follows. 

\smallskip

(3) As a commutative algebra $A^{(2)}$ is isomorphic to $\kk[M(3,1)]/(r)$  
for some $r$ in degree two. By part (2), the only possible 1-dimensional 
$\Gamma$-submodule in $(\kk[M(3,1)])_2$ is $M(1,1)$. The assertion follows.

\smallskip

(4) This follows from (3).
\end{proof}

\vspace{.4in}

\hspace{.4in} $\widetilde{\mathbb E}_8$:

\vspace{-.2in}

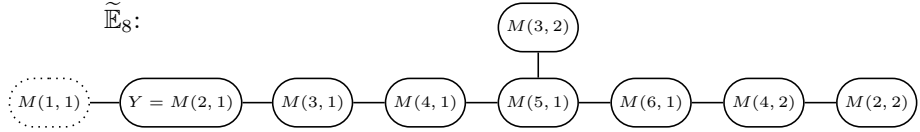
\begin{center}
\tikzset{elliptic state/.style={draw,ellipse}}
{\tiny 
\begin{tikzpicture}
[node distance=2cm, semithick,auto]
  \tikzstyle{every state}=[fill=white,draw=black, rounded rectangle, text=black]
  \node[state] at (0,0) (A)        {$M(4,1)$};
  \node[state] at (1.5,0) (B)        {$M(5,1)$};
  \node[state] at (-1.5,0) (C)        {$M(3,1)$};
  \node[state] at (1.5,1) (D)       {$M(3,2)$};
  \node[state] at (3,0) (E)      {$M(6,1)$};
  \node[state] at (-3.25,0) (F)      {$Y=M(2,1)$};
  \node[state,dotted] at (-5,0) (G)   {$M(1,1)$};
   \node[state] at (4.5,0) (H)      {$M(4,2)$};
     \node[state] at (6,0) (I)      {$M(2,2)$};

  \path (A) edge              node {} (B)
            edge              node {} (C)
        (B) edge              node {} (D)
              (B) edge              node {} (E)
        (C) edge              node {} (F)
        (E) edge              node {} (H)
         (H) edge              node {} (I)
        (F) edge              node {} (G);
\end{tikzpicture}
}
\end{center}

\bigskip

\begin{lemma}
\label{yylem3.3} 
Suppose the McKay quiver corresponding to $(A,H)$ is $\widetilde{\mathbb E}_8$. Then the following statements hold.

\begin{enumerate}
\item[(1)]
$\kk\langle M(3,1)\rangle_2=
M(3,1)\otimes M(3,1)=M(5,1)\oplus M(3,1)\oplus M(1,1)$ as $H$-modules.

\smallskip

\item[(2)]
$\kk[M(3,1)]_2=M(5,1)\oplus M(1,1)$ as $\Gamma$-modules.

\smallskip

\item[(3)]
There is a unique 1-dimensional subspace $\kk r$ of $\kk[M(3,1)]_2$, such that  $A^{(2)}$ is 
isomorphic to $\kk[M(3,1)]/(r)$ as graded $\Gamma$-module algebras.
In fact, $\kk r$ is the submodule $M(1,1)$ in part {\rm{(2)}}. 

\smallskip

\item[(4)]
For each $H$ of type $\widetilde{\mathbb E}_8$, $A^{(2)}$
is isomorphic to $\kk[u,v]^{(2)}$ as graded $\Gamma$-module algebras.
As a consequence, $A^H \cong \kk[u,v]^{\widetilde{\Gamma}}$ as graded 
algebras. \qed
\end{enumerate}
\end{lemma}

\subsection{Fixed subrings under Hopf actions of type $\mathbb{D}$}
\label{yysec3.2}
In this subsection we study type $\mathbb{D}$. In type $\mathbb{D}$ we have $A = \kk[u,v]$ 
or $\kk_{-1}[u,v]$ by Theorem~\ref{yythm0.4} and Proposition~\ref{yypro2.4}.
 Let $D_{2n}$ be the dihedral 
group of order $2n$ for some integer $n\geq 2$. One presentation of 
$D_{2n}$ is
$$\langle s_{+}, s_{-} \mid s_{+}^2=s_{-}^2=(s_{+}s_{-})^n=1\rangle.$$
The binary dihedral group of order $4n$, denoted by $\widetilde{D}_{2n}$, 
has a presentation
$$\langle \sigma, \tau \mid \sigma^{n}=\tau^2, \tau^4=1, \tau^{-1} \sigma 
\tau=\sigma^{-1}\rangle.$$
There is a short exact sequence of groups
$$1\to \{\tau^2, 1\}\to \widetilde{D}_{2n}\to D_{2n}\to 1$$
where $\{\tau^2, 1\}(\cong {\mathbb Z}/(2))$ is a central subgroup of 
$\widetilde{D}_{2n}$.
The map $\widetilde{D}_{2n}\to D_{2n}$ sends $\sigma$ to $s_{+}s_{-}$ and $\tau$ 
to $s_{+}$. We can embed $\widetilde{D}_{2n}$ to ${\text{SL}}_2(\kk)$ by setting 
\begin{equation}
\label{E3.3.1}\tag{E3.3.1}
\sigma \mapsto \begin{pmatrix} \omega & 0\\ 0 & \omega^{-1}\end{pmatrix}, \quad 
{\text{and}} \quad
\tau \mapsto \begin{pmatrix} 0 & i\\ i & 0\end{pmatrix}
\end{equation}
where $\omega$ is a primitive $2n$-th root of unity and $i^2=-1$. Using 
\eqref{E3.3.1}, $\widetilde{D}_{2n}$ acts on the commutative polynomial ring
$T:=\kk[u,v]$ naturally. The lemma below follows by an easy computation.

\begin{lemma}
\label{yylem3.4}  With the basis $\{u,v\}$ of $T_1$ and the action 
defined by \eqref{E3.3.1}, the following hold.
\begin{enumerate}
\item[(1)]
The $\sigma$ and $\tau$ actions on the basis $\{u^2, uv, v^2\}$ of $T_2$
are given by
\begin{equation}
\label{E3.4.1}\tag{E3.4.1}
\sigma = \begin{pmatrix} \omega^2 & 0&0\\ 0 & 1 & 0\\
0& 0 & \omega^{-2}\end{pmatrix} \quad 
{\text{and}} \quad
\tau= \begin{pmatrix} 0 & 0 & -1\\ 0 & -1& 0\\ -1 & 0 &0 \end{pmatrix}.
\end{equation}
So, $T_2$ is a direct sum of two 
$\widetilde{D}_{2n}$-submodules $\kk u^2+ \kk v^2$ and $\kk uv$.

\smallskip

\item[(2)]
Let $a=u^2, b=v^2$ and $c=uv$. Then $T^{(2)}=\kk [a,b,c]/(ab-c^2)$,
where the $D_{2n}$-action on $T^{(2)}$ is determined by
\begin{equation}
\label{E3.4.2}\tag{E3.4.2}
s_{+}= \begin{pmatrix} 0 & -1 & 0\\ -1 & 0& 0\\ 0 & 0 &-1 \end{pmatrix} \quad 
{\text{and}} \quad
s_{-} = \begin{pmatrix} 0& -\omega^{-2} &0\\ -\omega^{2} & 0 & 0\\
0& 0 & -1\end{pmatrix}. 
\end{equation}
\vspace{-.25in}

 \qed
\end{enumerate}
\end{lemma}

By \cite{Ma}, there are two different noncocommutative cocycle deformations of {$\kk{D}_{4n}$} 
(or of  $\kk \widetilde{D}_{2n}$), denoted by ${\mathcal A}(\widetilde{D}_{2n})^\circ$ and 
${\mathcal B}(\widetilde{D}_{2n})^\circ$. By  \cite[Proposition~2.5 and 
Theorem~4.4]{CDMM} and \cite[Theorem~4.1(1)]{Ma},
${\mathcal B}(\widetilde{D}_{2n})$ is self-dual 
and has no non-trivial cocycle twist (cotwist).
Here we will
 consider ${\mathcal B}(\widetilde{D}_{2n})$, 
or equivalently, ${\mathcal B}(\widetilde{D}_{2n})^{\circ}$.

\smallbreak By \cite[Definition 3.3(2)]{Ma}, for 
$n \geq 2$, ${\mathcal B}(\widetilde{D}_{2n})$ is the Hopf algebra 
generated by a central grouplike element $a$ of order $2$, and $s_{+}$ 
and $s_{-}$, subject to the relations
\begin{equation}
\label{E3.4.3}\tag{E3.4.3}
a^2=1, \quad a s_{\pm}=s_{\pm} a, \quad s_{\pm}^2=1, \quad (s_{+}s_{-})^n=a.
\end{equation}
with the structure
\begin{equation}
\label{E3.4.4}\tag{E3.4.4}
\Delta(s_{\pm })=s_{\pm}\otimes e_0 s_{\pm}+s_{\mp}\otimes e_1 s_{\pm},
\quad
\epsilon(s_{\pm})=1, 
\quad 
S(s_{\pm})=e_0s_{\pm}+e_{1}s_{\mp}
\end{equation}
where $e_{0}=\frac{1}{2}(1+a)$ and $e_1=\frac{1}{2}(1-a)$.

\smallbreak By \cite[Lemma 5.15, Remarks 5.22 and 5.23]{BN}, the Hopf algebra
${\mathcal B}(\widetilde{D}_{2n})$ has four non-isomorphic 1-dimensional
simple modules and $(n-1)$ non-isomorphic 2-dimensional 
simple modules. We list these 2-dimensional simple modules next.
Let $\xi$ be a primitive $2n$-th root of unity, and let $\omega=\xi^{l}$
for some $l$. Let $M(2,l)$, for $l=1,2,\dots, n-1$, be the 
2-dimensional simple left ${\mathcal B}(\widetilde{D}_{2n})$-module
determined by its action on a 2-dimensional 
$\Bbbk$-space $\kk s\oplus \kk t$ with matrix presentation
\begin{align}
\label{E3.4.5}\tag{E3.4.5}
s_{+}=\begin{pmatrix} 0& 1\\ 1& 0\end{pmatrix}, \quad
s_{-}=\begin{pmatrix} 0& \omega^{-1}\\ \omega& 0\end{pmatrix}, \quad
a=\begin{pmatrix} (-1)^l& 0\\ 0& (-1)^l\end{pmatrix}
\end{align}
where $\omega=\xi^l$. By using the relations of the algebra 
\eqref{E3.4.3} and the group characters, one can easily show that these 
$\{M(2,l)\}_{l=1}^{n-1}$ are non-isomorphic 2-dimensional simple modules. 
When $l$ is even, $M(2,l)$ is killed by the Hopf ideal
generated by $(a-1)$. Hence $M(2,l)$ is not inner faithful. When 
$l$ is odd and $g:=\gcd(l, 2n)>1$, then $M(2,l)$ is killed by the Hopf ideal
generated by $(s_{+}s_{-})^d-1, (s_{-}s_{+})^d-1$, where
$d=2n/g$. Hence $M(2,l)$ is not inner faithful. Therefore the inner faithful 
modules are precisely those $M(2,l)$, where $\gcd(l,2n)=1$ (or equivalently,
$\omega$ is primitive $2n$-th root of unity).

\smallbreak For every inner faithful 2-dimensional representation of 
${\mathcal B}(\widetilde{D}_{2n})$ (where $\omega=\xi^{l}$ 
with $\gcd(l,2n)=1$), we consider the action on the free algebra
$\kk \langle M(2,l)\rangle=\kk\langle s,t\rangle$ extended by 
\eqref{E3.4.5}. Using the
coproduct, these elements have the following 
matrix presentations when using the basis $\{s^2, st, ts, t^2\}$
of  $\kk \langle s,t\rangle_2$:

$$s_{+}=\begin{pmatrix} 0& 0& 0& \omega^{-1}\\ 0& 0&\omega^{-1}&0\\
0&\omega&0&0\\ \omega&0&0&0\end{pmatrix}, \quad
s_{-}=\begin{pmatrix} 0& 0& 0& \omega^{-1}\\ 0& 0&\omega&0\\
0&\omega^{-1}&0&0\\ \omega&0&0&0\end{pmatrix}, \quad
 a=\begin{pmatrix} 1& 0 & 0 & 0\\ 0& 1&0&0\\ 0& 0& 1& 0\\
0 & 0 & 0 & 1\end{pmatrix}.$$

\medskip

\begin{lemma}
\label{yylem3.5}
Suppose $n\geq 2$. Given the ${\mathcal B}(\widetilde{D}_{2n})$-action on $\kk\langle s,t\rangle$, we obtain   a unique 1-dimensional 
left trivial ${\mathcal B}(\widetilde{D}_{2n})$-module 
inside $\kk \langle s,t\rangle_2$,
which is spanned by $r:=\omega s^2+ t^2$.
\end{lemma}

\begin{proof} The subspace spanned by $st$ and $ts$ is a 
simple ${\mathcal B}(\widetilde{D}_{2n})$-module. The 
subspace spanned by $s^2$ and $t^2$ decomposes into 
$s_{\pm}$-eigenspaces with eigenvalues $1$ and $-1$. By a direct
computation, $r:=\omega s^2+ t^2$ is preserved by $s_{\pm}$ and
$a$. Hence $\kk r$ is a trivial ${\mathcal B}(\widetilde{D}_{2n})$-module.
\end{proof}

Let $Y$ be the 2-dimensional ${\mathcal B}(\widetilde{D}_{2n})$-module
determined by the action given in  \eqref{E3.4.5}.
By Lemma \ref{yylem3.5} and \cite[Theorem 2.1]{CKWZ}, there
is a unique graded ${\mathcal B}(\widetilde{D}_{2n})$-action 
on $\kk\langle s,t\rangle/(\omega s^2+ t^2)$ so that the degree one piece with the ${\mathcal B}(\widetilde{D}_{2n})$-action 
 is isomorphic to $Y$.
Since $\kk\langle s,t\rangle/(\omega s^2+ t^2)$ is isomorphic to 
$\kk_{-1}[u,v]$, there
is a unique graded ${\mathcal B}(\widetilde{D}_{2n})$-action 
on $\kk_{-1}[u,v]$ so that the degree one piece with the ${\mathcal B}(\widetilde{D}_{2n})$-action 
 is isomorphic to $Y$.

\smallbreak Let $\alpha=st$, $\beta=-\omega^{-1}ts$ and $\gamma= is^2$ (where
$i^2=-1$) in $B:=
\kk\langle s,t\rangle/(\omega s^2+ t^2)$. Then $B^{(2)}$ is a commutative
algebra generated by $\alpha$, $\beta$ and $\gamma$, and subject to one 
relation $\alpha\beta+\gamma^2=0$. With respect to the basis $\{\alpha,
\beta, \gamma\}$, $s_{\pm}$ have the following matrix presentation
(with $a$ being the identity)
\begin{equation}
\label{E3.5.1}\tag{E3.5.1}
s_{+} = \begin{pmatrix} 0& -1 &0\\ -1 & 0 & 0\\
0& 0 & -1\end{pmatrix}, \quad 
{\text{and}} \quad
s_{-}= \begin{pmatrix} 0 & -\omega^2 
& 0\\ -\omega^{-2} & 0& 0\\ 0 & 0 &-1 \end{pmatrix}.
\end{equation}
Now we are ready to prove the following.

\begin{lemma}
\label{yylem3.6}
Let $n\geq 2$  and let $H:={\mathcal B}(\widetilde{D}_{2n})\cong 
{\mathcal B}(\widetilde{D}_{2n})^{\circ}$ act on $A=\kk_{-1}[u,v]$
inner faithfully with trivial homological determinant. Then $A^{(2)}$
is isomorphic to $\kk[u,v]^{(2)}$ as graded $D_{2n}$-module algebras.
So, $A^H\cong \kk[u,v]^{\widetilde{D}_{2n}}$ as graded 
algebras. 
\end{lemma}

\begin{proof} By Lemma \ref{yylem3.5}, such an $H$-action on
$A$ is uniquely determined by $\omega:=\xi^l$. 
The induced $D_{2n}$-action on $A^{(2)}\cong 
\kk[\alpha,\beta,\gamma]/(\alpha\beta+\gamma^2)$ is determined by
\eqref{E3.5.1}. The main assertion follows by comparing \eqref{E3.5.1}
with \eqref{E3.4.2}. 

\smallbreak As a consequence
\begin{equation}
\label{E3.6.1}\tag{E3.6.1}
(A^{(2)})^{D_{2n}}\cong (\kk[u,v]^{(2)})^{D_{2n}}.
\end{equation}
The final statement follows by \cite[Lemma 6.27(b)]{CKWZ} and \eqref{E3.6.1}.
\end{proof}

\section{Auslander's theorem for quantum binary polyhedral groups} 
\label{yysec4}

In this section we prove Conjecture~\ref{yycon0.2} for every pair
$(A,H)$ in Theorem~\ref{yythm0.4}.

\begin{theorem} 
\label{yythm4.1} 
Conjecture~{\rm{\ref{yycon0.2}}} is true for AS regular algebras of 
dimension 2.
\end{theorem}

Some partial results on Auslander's Theorem for group actions on AS 
regular algebras were obtained by  Mori-Ueyama in \cite{MU}. For $A$ 
of dimension 2 we complete the remaining cases where the action is a 
group action, as well as handle the cases where the Hopf algebra is 
not a group algebra. The proof consists of the following case-by-case 
analysis. We start with a well-known fact, which is a special case of 
the original Auslander's Theorem.

\subsection{Auslander's Theorem: Case ({\tt a}):}
\label{yysec4.1}

\begin{proposition} \cite[Proposition~3.4]{Aus62}
\label{yypro4.2} 
Let $(A,H)$ be in case {\rm ({\tt a})} of Theorem {\rm{\ref{yythm0.4}}}. 
Then Conjecture {\rm{\ref{yycon0.2}}} holds. \qed
\end{proposition}

The following cases were proved by Mori-Ueyama \cite{MU}.

\subsection{Auslander's Theorem: Cases ({\tt b}), ({\tt g}), and ({\tt h}):}
\label{yysec4.2}

\begin{proposition}
\label{yypro4.3} \cite[Theorem 4.5(1)]{MU}
Let $(A,H)$ be in cases {\rm  ({\tt b})},{\rm  ({\tt g})},{\rm  ({\tt h})} of Theorem~{\rm{\ref{yythm0.4}}}. 
Then Conjecture {\rm{\ref{yycon0.2}}} holds. \qed
\end{proposition}

\subsection{Auslander's Theorem: Cases ({\tt c}) and ({\tt d}):}
\label{yysec4.3}

To complete the next cases we need a result from \cite{BHZ}, which is a 
generalization of \cite[Theorem 3.10]{MU}. Throughout let $t$ be an 
integral of a semisimple Hopf algebra $H$. 

\begin{lemma}
\label{yylem4.4} 
\cite[Theorem 0.3]{BHZ} 
Let $A$ be a noetherian AS regular algebra of dimension 2,
and $H$ act on $A$ inner faithfully and homogeneously.
Let $e=1\# t \in A\# H$. The following conditions are equivalent
\begin{enumerate}
\item[(1)]
$(A\# H)/(e)$ is finite dimensional.

\smallskip

\item[(2)]
The natural algebra homomorphism
$$A\# H\to \End_{(A^H)^{\op}} (A)$$
is an isomorphism. Namely, \eqref{E0.2.1} holds. \qed
\end{enumerate} 
\end{lemma}

\begin{proposition}
\label{yypro4.5}
Let $(A,H)$ be as in Theorem {\rm{\ref{yythm0.4}}}{\rm{({\tt c})}}. 
Then Conjecture {\rm{\ref{yycon0.2}}} holds.
\end{proposition}

\begin{proof} In this case $A=\kk_{-1}[u,v]$ and $G = C_2$. Let $I$ be the ideal $(e)$. By Lemma \ref{yylem4.4},
it suffices to show that $(A\# \kk C_2)/I$ is finite dimensional. 

\smallbreak We claim that $u^2\# 1, v^2\# 1\in I$. Supposing that the claim 
holds, since $A_{\geq 3}$ is in the ideal generated by $u^2, v^2$, 
then $I$ contains $(A\# \kk C_2)_{\geq 3}$. Therefore $(A\# \kk C_2)/I$ 
is finite dimensional, as desired.

\smallbreak To prove the claim, write $C_2=\{1,g\}$, where $1$ is the identity 
of $C_2$ and $g^2=1$, and, replacing $e$ by $2e$, it suffices to let $e := 1 \#(1+g)$. Using the notation in Theorem \ref{yythm0.4}, 
$g(u)=v$ and $g(v)=u$, $e (u\# 1)=u\# 1+ v\# g$ and 
$(v\# 1)e =v\# 1+ v\# g$, and hence, both are in $I$.  Hence $ (u-v)\# 1\in I$. 
Now $(u\# 1)((u-v)\# 1)=(u^2-uv)\# 1$ and $((u-v)\# 1) (u\# 1)=
(u^2+ uv)\# 1$ are in $I$. Thus $u^2\# 1\in I$, and by symmetry,
$v^2\# 1\in I$. Therefore we have verified the claim, and, consequently,
the assertion of the proposition.
\end{proof}

The next lemma provides some of the computation needed in the proof 
of case~({\tt d}), which is similar to case ({\tt c}), but more complicated.   In case  ({\tt d}) $A=\kk_{-1}[u,v]$ and $H =\kk D_{2n}$.
Let 
$$\sigma:=\begin{pmatrix} 
0& 1\\ 1& 0\end{pmatrix} \quad \text{and} \quad g:=\begin{pmatrix}
\xi & 0\\ 0& \xi^{-1}\end{pmatrix}$$ where $\xi$ 
is a primitive $n$-th root of unity. The dihedral group $D_{2n}$ of order $2n$
is generated by $\sigma$ and $g$. Denote
$$\begin{aligned}
G_d &= 1+(\xi)^d g+ (\xi)^{2d} g^2 +\cdots (\xi)^{(n-1)d} g^{n-1},\\
F_d &= G_d \sigma.
\end{aligned}
$$
Recall that $e=1\#(G_0+F_0)$ and $I$ is the ideal $(e)$. 

\begin{lemma}
\label{yylem4.6} 
Retain the  notation above and assume the hypotheses as in case 
{\rm({\tt d})} of Theorem {\rm{\ref{yythm0.4}}}.
\begin{enumerate}
\item[(1)]
$G_{d}=G_{n+d}$ and $F_{d}=F_{n+d}$ for all $d$.

\smallskip

\item[(2)]
$(1\# G_d)(u^i\# 1)=u^i\# G_{d+i}$ and $(1\# G_d)(v^i\# 1)= v^i\# G_{d-i}$
for all $i$.

\smallskip

\item[(3)]
$(1\# F_d)(u^i\# 1)=v^i\# F_{d-i}$ and $(1\# F_d)(v^i\# 1)= u^i\# F_{d+i}$
for all $i$.

\smallskip

\item[(4)]
$e (u^i\# 1)=u^i \# G_i+ v^i \# F_{-i}$ for all $i$.

\smallskip

\item[(5)]
$(u^{n-i}\# 1) e (u^i\#1)= u^n\#G_i+ u^{n-i} v^{i}\# F_{-i}$ for all $i$.

\smallskip

\item[(6)]
$(v^i\#1) e (u^{n-i}\# 1)=v^i u^{n-i}\# G_{-i}+ v^n \# F_i$ for all $i$.

\smallskip

\item[(7)]
$(u^n-(-1)^{i(n-i)}v^n)\# G_i\in I$ for all $i$.

\smallskip

\item[(8)]
$e(uv\# 1)=uv \# G_0-uv \# F_0$.

\smallskip

\item[(9)]
$u^{n+1} v\# 1\in I$, and by symmetry, $uv^{n+1}\# 1\in I$.

\smallskip

\item[(10)]
If $n$ is even, $(u^{2n}-v^{2n})\# 1\in I$. If $n$ is
odd, $(u^{2n}+v^{2n})\# 1\in I$. 

\smallskip

\item[(11)]
$u^{2n+1}\# 1\in I$, and by symmetry, $v^{2n+1}\# 1\in I$.
\end{enumerate}
\end{lemma}

\begin{proof}
(1) This holds because $\xi$ is an $n$-th root of unity.

\smallskip

(2) This follows from the equations 
$$(1\# (\xi^d)^n g^n)(u^i\# 1)
=u^i \# (\xi^{d+i})^n g^n \quad
{\text{and}} \quad (1\# (\xi^d)^n g^n)(v^i\# 1)
=v^i \# (\xi^{d-i})^n g^n$$ 
for all $i,d,n$.

\smallskip

(3) The proof is similar to (2).

\smallskip

(4) By parts (2,3), for all $i$,
$$e(u^i\# 1)=(1\# G_0) (u^i\# 1)+ (1\#F_0)(u^i\# 1)=
u^i\# G_i +v^i\# F_{-i}.$$

(5) This follows from part (4).

\smallskip

(6) This follows from part (4), replacing $i$ by
$n-i$.

\smallskip

(7) By parts (5,6), and the fact that $uv = -vu$,
$$\begin{aligned}
(u^n-&(-1)^{i(n-i)}v^n)\# G_i\\
&=u^n\#G_i+ u^{n-i} v^{i}\# F_{-i}
-(-1)^{i(n-i)}
(v^i u^{n-i}\# G_{-i}+ v^n \# F_i)(1\# \sigma)\\
&=
(u^{n-i}\# 1) e (u^i\#1)-(-1)^{i(n-i)}(v^i\#1) e (u^{n-i}\# 1)
(1\# \sigma) ~~\in I.
\end{aligned}
$$

(8) This is easy to check.

\smallskip

(9) Note that $(uv\# 1)e= uv\# G_0+uv\# F_0$. Combining
this with part (8), we have $uv\# G_0, uv\# F_0\in I$.
Now, for every $i$,
$$(-1)^i u^{n+1} v\# G_i =u^{n-i+1} v u^i 
\# G_i=(u^{n-i}\# 1) (uv \# G_0) (u^i \# 1)\in I.$$
Since $\sum_{i=0}^{n-1} G_i=n$, $u^{n+1} v\# 1\in I$.
By symmetry, $uv^{n+1}\# 1\in I$.

\smallskip

(10) First we assume $n$ is even, and then $u^n$ commutes with
$v^n$. By part (7),
$$(u^{2n}-v^{2n})\# G_i=
((u^n+(-1)^{i(n-i)}v^n)\# 1) (u^n-(-1)^{i(n-i)}v^n)\# G_i) 
\in I,$$
for all $i$, which implies that $(u^{2n}-v^{2n})\# 1\in I$.

Second we assume that $n$ is odd. By parts (1,2), $1\# G_i$ 
commutes with $u^n\# 1$ and $v^n \#1$. Since $u^n v^n =-
v^n u^n$, 
$$((u^n-(-1)^{i(n-i)}v^n)\# G_i)^2=
(u^n-(-1)^{i(n-i)}v^n)^2\# (G_i)^2=
(u^{2n}+v^{2n})\# nG_i.$$
By part (7), $(u^{2n}+v^{2n})\# G_i\in I$ for all $i$. 
Hence $(u^{2n}+v^{2n})\# 1\in I$. 

\smallskip

(11) This follows from parts (9,10).
\end{proof}

\begin{proposition}
\label{yypro4.7}
Let $(A,H)$ be as in 
Theorem {\rm{\ref{yythm0.4}}}{\rm{({\tt d})}}. Then Conjecture 
{\rm{\ref{yycon0.2}}} holds.
\end{proposition} 
\begin{proof}
By Lemma \ref{yylem4.6}(11), $u^{2n+1}\# 1, v^{2n+1}\# 1\in I$. 
Since $A_{\geq 4n+2}$ is in the ideal generated by $u^{2n+1}$ and
$v^{2n+1}$, $I$ contains $(A\# \kk G)_{\geq 4n+2}$.
Therefore $(A\# \kk G)/I$ is finite dimensional, as desired.
\end{proof}

\subsection{Auslander's Theorem: Case ({\tt e}):}
\label{yysec4.4}

To handle case ({\tt e}) in Theorem \ref{yythm0.4}, we introduce some 
different techniques. Let $H$ be the dual of the group algebra 
$\kk G$ (where $G=D_{2n}$). Let $\{g_1,\dots, g_m\}$ be the set 
$G$, where $g_1$ is the unit of $G$. The dual basis 
$\{p_{g_1},\dots, p_{g_{m}}\}$
is the complete set of orthogonal idempotents of $H$ that
forms a $\kk$-linear basis of $H$. It is well-known that
$p_{g_1}$ is an integral of $H$, and $\sum_{s=1}^{m} p_{g_{s}}$
is the algebra identity of $H$. The coproduct of $H$ is 
determined by
$$\Delta(p_g)=\sum_{h \in G} p_{h}\otimes p_{h^{-1} g}$$
for all $g\in G$. The (left) $H$-action on $A$ is equivalent to the
(right) $\kk G$-coaction on $A$, that is equivalent to a $G$-grading 
on $A$. Therefore $A$ has a natural $G$-grading induced by 
the $H$-action. For a $G$-homogeneous element $f$ in $A$, its
$G$-grading is denoted by $\deg_G(f)$. For the next two results, 
set $e=1\# t=1\# p_{g_1}$ and $I=(e)$. 

\begin{lemma}
\label{yylem4.8} Retain the notation above.
\begin{enumerate}
\item[(1)]
Let $f$ be a $G$-homogeneous element of $A$ of degree $g\in G$. 
Then  $$e(f\otimes 1)=f\otimes p_{g^{-1}}.$$
\item[(2)]
Suppose $f_1,\dots, f_m$ are homogeneous elements in $A$
such that the set 
$$\{\deg_G(f_d)\cdots \deg_G(f_m)\}_{d=1}^m$$ 
equals $G$. Then the product $(f_1\cdots f_m) \# 1\in I$.
\end{enumerate}
\end{lemma}

\begin{proof} (1) By definition,
$$\begin{aligned}
e(f\# 1)=(1\# p_{g_1})(f\# 1)=\sum_{h\in G} p_{h}(f)\# p_{h^{-1}}
=p_{g}(f)\# p_{g^{-1}}=f\# p_{g^{-1}}.
\end{aligned}
$$

(2) Let $h_d=\deg_G(f_d)\cdots \deg_G(f_m)$. By part (1),
$$ e((f_d\cdots f_m)\# 1)=(f_d\cdots f_m) \# p_{h_d^{-1}}\in I$$
for all $d$. Hence $(f_1\cdots f_m) \# p_{h_d^{-1}}\in I$
for all $d$. By hypothesis, $\{h_d\}_{d=1}^m$ equals $G$.
Then $1=\sum_{d=1}^m p_{h_d^{-1}}$, which implies
that $(f_1\cdots f_m)\# 1\in I$.
\end{proof}

Now we are ready to prove Theorem \ref{yythm4.1} for case ({\tt e}).

\begin{proposition}
\label{yypro4.9}
Let $(A,H)$ be in 
Theorem {\rm{\ref{yythm0.4}}}{\rm{({\tt e})}}. Then Conjecture~{\rm{\ref{yycon0.2}}} holds.
\end{proposition}

\begin{proof} As in the statement of Theorem \ref{yythm0.4}({\tt e}),
we recycle the variables $u$ and $v$ to denote the generators for $A$ 
satisfying the relation $u^2+v^2=0$. (These are $v_1$ and $v_2$ in 
Theorem \ref{yythm0.4}.) Write $G:=D_{2n}$ as 
$\langle g,h \mid g^2=h^2=(gh)^n=1\rangle$
and $\deg_G u=g$ and $\deg_G v=h$. 

\smallbreak For $1\leq i \leq m:=2n$, let $f_i$ be either $u$ if $i$ is odd, or $v$ if $i$ is even. Then the hypotheses of Lemma \ref{yylem4.8}(2)
hold. As a consequence, $f_1\cdots f_m\# 1\in I$, or $(u v)^n\# 1\in I$.
By symmetry, $(vu)^n\# 1\in I$. Since $(uv)^n (vu)^n=u^{2n}v^{2n}
=(-1)^n u^{4n}$, we have $u^{4n}\# 1\in I$. Since every element
in $A$ is a linear combination of $u^{a} (v u)^{b} v^{c}$
where $c$ is either 0 or 1, $I$ contains the ideal
$(A\# H)_{\geq (6n+1)}$. This implies that $(A\# H)/I$ is finite
dimensional. The assertion follows by Lemma \ref{yylem4.4}.
\end{proof}

\subsection{Auslander's Theorem: Case ({\tt f}):}
\label{yysec4.5}

It remains to consider case ({\tt f}). We need to understand the structure
of the fixed subrings, and we will use the results in Section \ref{yysec3}. 
Let $A=\kk_{-1}[u,v]$ as in  case ({\tt f}),
$H$ be one of the Hopf 
algebras in case~({\tt f}), and let $B=A^H$. Then $A$ is a left and a right
$B$-module. We will be using the right $B$-module structure on $A$,
which is equivalent to a left $B^{\op}$-module structure of $A$.
The following lemma is well-known.

\begin{lemma} \cite[Lemma 3.10]{BHZ}
\label{yylem4.10} Retain the notation above. If $A\# H$
is prime, then the natural algebra map 
$\theta: A\# H\to \End_{B^{\op}}(A)$ is injective. \qed
\end{lemma}

Moreover, we will use the following well-known lemma in the results below.

\begin{lemma} 
\label{yylem4.11} 
Let $R$ be a connected graded algebra admitting an action by a semisimple 
Hopf algebra $H$. Let $M$ be a graded, bounded below, left $R\# H$-module. 
A minimal projective resolution of the $R\# H$-module $M$ 
is also a minimal free resolution of left $R$-module $M$. \qed
\end{lemma}

Let $H$ be the {Hopf algebra} $({\mathcal A}(\widetilde{D}_{2m}))^{\circ}$, 
which is one of the Hopf algebras of type $\mathbb{D}$ in case ({\tt f}). Note 
that there is other notation for the dicyclic group (or binary 
dihedral group) $\widetilde{D}_{2m}$; in \cite{Ma} it is denoted 
by $T_{4m}$; in \cite{BN}, it is denoted by $\widetilde{D}_{m}$;
in other papers it is denoted by $BD_{4m}$. By \cite[Theorem 4.1(2)]{Ma}, $H$ is a {\it dual cocycle 
twist} of the group algebra $\kk {{D}_{4m}}$. A dual cocycle twist is also 
called a {\it cotwist} in \cite[Definition 2.5]{Mo2}. The next lemma lists 
some general facts about cotwists, some of which is from Montgomery \cite{Mo2}, Davies \cite{Dav1,Dav2}, or Chirvasitu-Smith \cite{CS}.

\begin{lemma}
\label{yylem4.12}
Let $R$ be a connected graded algebra and $H$ be a semisimple
Hopf algebra such that $R$ is a graded $H$-module algebra. 
Let $\Omega\in H\otimes H$ be a dual cocycle \cite[Section 2]{Mo2}, 
and $H^{\Omega}$ be a cotwist of $H$ with respect to $\Omega$
\cite[Definition 2.5]{Mo2}.
Then the following statements hold.
\begin{enumerate}
\item[(1)]
$H$ is a cotwist $(H^{\Omega})^{\Omega^{-1}}$ of $H^{\Omega}$. 

\smallskip

\item[(2)]
The twist of $R$, denoted by $R_{\Omega}$, 
as defined in \cite[Definition 2.6]{Mo2},
is a graded $H^{\Omega}$-module algebra. 

\smallskip

\item[(3)]
There is a graded algebra isomorphism
$$\phi: R_{\Omega}\# H^{\Omega}\longrightarrow  R\# H$$
defined by
$a\# h\mapsto \sum \Omega^1 \cdot a\# \Omega^2 h$
for all $a\in R$ and $h\in H$.

\smallskip

\item[(4)]
There is an equivalence of graded module categories
$$F_{\phi}: \quad 
R_{\Omega}\# H^{\Omega} \mhyphen \Mod\cong R\# H \mhyphen \Mod$$
induced by the isomorphism $\phi$ in part {\rm{(3)}}.

\smallskip

\item[(5)]
If $R$ is noetherian, then so is $R_{\Omega}$.

\smallskip

\item[(6)]
$\gldim R=\gldim R_{\Omega}$ and $\GKdim R=\GKdim R_{\Omega}$.

\smallskip

\item[(7)]
$R$ is AS regular if and only if $R_{\Omega}$ is AS regular.

\smallskip

\item[(8)]
Suppose $R$ is AS regular.
If the $H$-action on $R$ 
has trivial homological determinant, then so does 
the $H^{\Omega}$-action on $R_{\Omega}$.

\smallskip

\item[(9)]
Let $R$ and $R_{\Omega}$ be locally finite graded domains.
If $R\# H\to \End_{(R^H)^{\op}}(R)$ is an isomorphism, 
then $R_{\Omega}\# H^{\Omega}\to 
\End_{(R_{\Omega}^{H^{\Omega}})^{\op}}(R_{\Omega})$
is an isomorphism.  
\end{enumerate}
\end{lemma}

\begin{proof} (1) This is well-known.

\smallskip

(2,3) See \cite[Section~2]{Mo2} and references therein.

\smallskip

(4) 
For every $M\in R\# H \mhyphen\Mod$, define the left 
$R_{\Omega}\# H^{\Omega}$-module structure on $M$, denoted by $M_{\Omega}$, 
by
$$(a\# h) \bullet_{\Omega} m=\sum ((\Omega^1 \cdot a)
\# \Omega^2 h) m$$
for all $a\in R$ and $h\in H$ and $m\in M$.  The result then follows from part (3).

\smallskip

(5) See \cite[Proposition~3.1]{Mo2}.

\smallskip

(6)
Consider a minimal projective resolution of the 
left trivial module $\kk$ over $R_{\Omega}\# H^{\Omega}$:
\begin{equation}
\label{E4.12.1}\tag{E4.12.1}
\cdots \to P_2\to P_1\to  P_0 \to \kk\to 0
\end{equation}
where $P_0=R_{\Omega}$ as a left $R_{\Omega}\# H^{\Omega}$-module. 
By Lemma \ref{yylem4.11}, \eqref{E4.12.1} 
is a minimal projective (free) resolution of the trivial 
$R_{\Omega}$-module $\kk$. Applying the equivalence $F_{\phi}$ of 
categories in part (4), we obtain 
\begin{equation}
\label{E4.12.2}\tag{E4.12.2}
\cdots \to F_{\phi}(P_2)\to F_{\phi}(P_1)\to  
F_{\phi}(P_0) \to \kk\to 0
\end{equation}
which is a minimal projective resolution of the trivial $R\# H$-module 
$\kk$. By Lem-ma~\ref{yylem4.11}, \eqref{E4.12.2} is a minimal projective 
(free) resolution of the trivial $R$-module $\kk$. Combining these
facts, we have that the shape of the minimal free
resolution of the trivial $R$-module $\kk$ is exactly the
shape of the minimal free resolution of the trivial 
$R_{\Omega}$-module $\kk$. From this, we have the following
consequences:
\begin{enumerate}
\item[(i)]
$\gldim R=\gldim R_{\Omega}$.
\item[(ii)]
$R$ is Koszul (respectively, $N$-Koszul) 
if and only if $R_{\Omega}$ is.
\item[(iii)]
$R$ and $R_{\Omega}$ have the same Hilbert series.
\item[(iv)]
$R$ and $R_{\Omega}$ have the same Gelfand-Kirillov dimension. 
\end{enumerate}

\smallskip

(7) Following part (4), $F_\phi$ maps $R_\Omega$ to $R$, and 
by \cite[Lemma 5.2]{KKZ2}
and part~(4),
we have
$$\Ext^i_R(\kk, R)\cong \Ext^i_{R\# H}(\kk\# H, R)
\cong \Ext^i_{R_{\Omega}\# H^{\Omega}}(\kk\# H^{\Omega}, R_{\Omega})
\cong \Ext^i_{R_{\Omega}}(\kk, R_{\Omega}).
$$
With part~(6), we obtain  that $R$ is AS regular if and only if 
$R_{\Omega}$ is.

\smallskip

(8) 
Let $n=\gldim R$. Suppose $R$ (or $R_{\Omega}$) is AS
regular. 
Since $F_\phi$ maps $R_\Omega$ to $R$ 
(and $F_\phi$ commutes with the shift), $P_n \cong R_{\Omega}(\ell)$
as a graded left $R_\Omega \# H^\Omega$-module.
It follows from Definition \ref{yydef1.6} that 
the $H$-action on $R$ has trivial
homological determinant if and only if $F_{\phi}(P_n)\cong
R(\ell)$ as a graded  left $R\# H$-module, or the $H$-action on the
lowest degree of $F_{\phi}(P_n)$ is trivial. This implies
that the $H^{\Omega}$-action on $R_{\Omega}$ has trivial
homological determinant. 

\smallskip

(9) 
Since $R\# H\to \End_{(R^H)^{\op}}(R)$ is an isomorphism, 
$R\# H$ is prime by \cite[Lemma~3.10]{BHZ}. By part (3), 
$R_{\Omega}\# H^{\Omega}$ is prime. Hence, by Lemma 
\ref{yylem4.10}, the map $R_{\Omega}\# H^{\Omega}\to 
\End_{(R_{\Omega}^{H^{\Omega}})^{\op}}(R_{\Omega})$ is injective.
It remains to show that these algebras have the same Hilbert series.
It is clear that $R\# H$ and $R_{\Omega}\# H^{\Omega}$
have the same Hilbert series.  By the hypothesis, it suffices to show that
$\End_{(R^{H})^{\op}}(R)$
and $\End_{(R_{\Omega}^{H^{\Omega}})^{\op}}(R_{\Omega})$
have the same Hilbert series. By definition,
$R=R_{\Omega}$ as vector spaces and the $H$-action on $R$
agrees with the $H^{\Omega}$-action on $R_{\Omega}$. So $R^H=R_{\Omega}^{H^{\Omega}}$
as vector spaces and as algebras. Further,
it is easy to check that the left $R^H$-multiplication on
$R$ is exactly the left $R_{\Omega}^{H^{\Omega}}$-multiplication on
$R_{\Omega}$. As a consequence, $R_{\Omega}$, as a left 
$R_{\Omega}^{H^{\Omega}}$-module, is equal to $R$, as a left
$R^H$-modules. Hence $\End_{(R^{H})^{\op}}(R)=
\End_{(R_{\Omega}^{H^{\Omega}})^{\op}}(R_{\Omega})$, as required. 
Combining these facts, we obtain that
$R_{\Omega}\# H^{\Omega}\cong
\End_{(R_{\Omega}^{H^{\Omega}})^{\op}}(R_{\Omega})$, as desired.
\end{proof}

Now we prove part of Theorem~\ref{yythm0.4} for case ({\tt f}).

\begin{proposition}
\label{yypro4.13}
Let $H=({\mathcal A}(\widetilde{D}_{2m}))^{\circ}$ and $A=\kk_{-1}[u,v]$ be 
as in case {\rm{({\tt f})}} of Theorem {\rm{\ref{yythm0.4}}}. Then Conjecture 
{\rm{\ref{yycon0.2}}} holds.
\end{proposition}

\begin{proof} By Lemma \ref{yylem4.12}, it suffices to show that Conjecture 
\ref{yycon0.2}  holds for some cotwist of $H$. By \cite[Theorem 4.1]{Ma}, $H$ is a cotwist $(\kk D_{2m})^{\Omega^{-1}}$ of
 $\kk{D}_{2m}$. Further, $A_\Omega$ is noetherian AS regular of dimension~2 and admits an $H^\Omega$-action with trivial homological determinant by Lemma~\ref{yylem4.11}(2,5,6,7,8). Since $H^\Omega = \kk D_{2m}$, we have by Theorem~\ref{yythm0.4}({\tt d}) that $A_\Omega \cong A$ as graded $H^\Omega$-module algebras. Hence, Conjecture~\ref{yycon0.2} 
holds by case ({\tt d}).
\end{proof}

We consider the situation where $H={\mathcal D}(\widetilde{\Gamma})^\circ$
for the other algebras in case ({\tt f}) of Theorem \ref{yythm0.4}. There is 
a short exact sequence of Hopf algebras
\begin{equation}
\label{E4.13.1}\tag{E4.13.1}
\kk \to \kk {\mathbb Z}_2 \to H\to \kk \Gamma \to \kk,
\end{equation}
and by \cite[Lemma 6.27(b)]{CKWZ}, $A^{H}=(A^{(2)})^{\Gamma}$, where 
$A=\kk_{-1}[u,v]$. There is a canonical Hopf algebra map $\gamma: H\to \kk\Gamma$ 
given in \eqref{E4.13.1}. Let $t$ be the integral of $H$ such that 
$\epsilon(t)=1$. Then  $t': =\gamma(t)$ is the integral of $\kk \Gamma$ such 
that $\epsilon(t')=1$.

\begin{lemma}
\label{yylem4.14} 
Retain the notation above.
Let $(1\# t)$ and $(1\# t')$ be the 2-sided ideal generated
by $1\# t$ and $1\# t'$, respectively. Then 
$(A\# H)/(1\# t)$ is finite dimensional if and only if 
$(A^{(2)}\# \kk\Gamma)/(1\# t')$ is finite dimensional.
\end{lemma}

\begin{proof} By \eqref{E4.13.1}, $H/(g-1)\cong \kk \Gamma$,
where $g\in H$ is a central grouplike element of order $2$. Let 
$\{f_1:=\frac{1}{2}(1+g), f_2,\dots, f_m, h_1,\dots, h_m\}$ be 
a $\kk$-basis of $H$ such that: 
\begin{enumerate}
\item[(i)]
$h_i\in (g-1)H$ for all $i$, 
\item[(ii)]
$g f_i=f_i$ for all $i$, and 
\item[(iii)]
$\{\gamma(f_i)\}_{i=1}^m$ is the set of grouplike elements in $\Gamma$. 
\end{enumerate}
Suppose $(A\# H)/(1\# t)$ is finite dimensional. Then 
$u^{2d}, v^{2d}$ are in the ideal $I$ of $A\# H$ generated by 
$1\# t$. Write
$$u^{2d}\# 1=\sum x_i (1\# t) y_i=
\sum_{\deg x_i=odd}x_i (1\# t) y_i+\sum_{\deg x_i=even}
x_i (1\# t) y_i$$ 
where $x_i$ and $y_i$ are homogeneous in $A\# H$ and $\deg (x_i)+
\deg (y_i)=2d$ for all $i$. Multiplying the above equation by $f_1=\frac{1}{2}(1+g)$ 
 we obtain
$$\begin{aligned}
u^{2d}\# f_1 &= (1\# f_1)(u^{2d}\# 1)=
(1\# \textstyle\frac{1}{2}(1+g)) (u^{2d}\# 1)\\
&=\sum_{\deg x_i=odd}(1\# \textstyle\frac{1}{2}(1+g)) x_i (1\# t) y_i+\displaystyle \sum_{\deg x_i=even}
(1\# f_1)x_i (1\# t) y_i\\
&=\sum_{\deg x_i=odd}x_i (1\# \textstyle\frac{1}{2}(1-g)) (1\# t) y_i+\displaystyle\sum_{\deg x_i=even}
x_i (1\# f_1)(1\# t) y_i\\
&=\sum_{\deg x_i=even} x_i (1\# t) y_i.
\end{aligned}
$$
Let $\gamma$ denote also the induced surjective algebra
homomorphism $A^{(2)}\# H\to A^{(2)}\# \Bbbk \Gamma$.
Applying $\gamma$ to the above, we obtain 
$$u^{2d}\# 1=\sum_{\deg x_i=even} \gamma(x_i) (1\# t') \gamma(y_i),$$
which holds in $A^{(2)}\# \kk \Gamma$. Thus $u^{2d}\# 1$ is in the ideal $I'$ of $A^{(2)}\# \kk\Gamma$ 
generated by $1\# t'$. Similarly, $v^{2d}\# 1$ is in the ideal $I'$ of 
$A^{(2)}\# \kk \Gamma$ generated by $1\# t'$. This implies that $(A^{(2)}\# \kk \Gamma)/I'$ is finite 
dimensional.

\smallbreak Conversely, suppose that $(A^{(2)}\# \kk \Gamma)/I'$ is finite  dimensional. Then
there is an integer $d\geq 1$ such that  
$u^{2d}\# 1$ and $ v^{2d}\# 1$ are in the ideal $I'$ generated by $1\# t'$.
Write
$$u^{2d}\# 1=\sum_{i} \gamma(x_i) (1\# t') \gamma(y_i)$$
where $x_i,y_i$ are elements in $A^{(2)}\# H$. Let $F$ be the
element $(\sum_{i} x_i (1\# t) y_i)(1\# f_1)$ in $A\# H$. Then 
$$\gamma(F)=\sum_{i} \gamma(x_i) (1\# t') \gamma(y_i)=u^{2d}\# 1.$$
Write $F=\sum z_i\# f_i+\sum t_i\# g_i$. Since $g_i\in (g-1)$ and 
$F=F(1\# g)$, we have $t_i=0$. Hence 
$u^{2d}\# 1=\gamma(F)=\sum z_i \# \gamma(f_i)$.
Since $\{\gamma(f_i)\}_{i=1}^m$ are linearly independent, $z_1=u^{2d}$ 
and $z_i=0$ for all $i\neq 1$. Thus $F=u^{2d}\# f_1$. 
Therefore, we have that $(u\# 1) F+ F (u\# 1)=u^{2d+1}\# 1$, 
which must be in $I$. Similarly, $v^{2d+1}\# 1\in I$. This implies that 
$(A\# H)/I$ is finite dimensional.
\end{proof}

If $(A,H)$ is in the case of Theorem {\rm{\ref{yythm0.4}}}{\rm{({\tt f})}} then $A=\kk_{-1}[u,v]$ and $H=({\mathcal A}(\widetilde{D}_{2m}))^{\circ}$ or equation \ref{E4.13.1} holds. 
\begin{proposition}
\label{yypro4.15}
Let $(A,H)$ be in the case  of Theorem {\rm{\ref{yythm0.4}}}{\rm{({\tt f})}}. 
Then Conjecture {\rm{\ref{yycon0.2}}} holds.
\end{proposition}

\begin{proof} When $H=({\mathcal A}(\widetilde{D}_{2m}))^{\circ}$, 
the assertion is Proposition \ref{yypro4.13}. For the other cases,
by Lemmas \ref{yylem4.4} and \ref{yylem4.14}, it suffices to
show that $(A^{(2)}\# \kk \Gamma)/(1\# t')$ is finite dimensional.
By Lemmas \ref{yylem3.1}(4), \ref{yylem3.2}(4), \ref{yylem3.3}(4)
and \ref{yylem3.6}, it is equivalent to show that
$(\kk[u,v]^{(2)} \# \kk \Gamma)/(1\# t')$ is finite
dimensional. This last statement follows by Proposition~\ref{yypro4.2}, Lemma~\ref{yylem4.4},
and the commutative version of Lemma \ref{yylem4.14}.
\end{proof}

\begin{proof}[Proof of Theorem \ref{yythm4.1}]
It follows by combining Propositions \ref{yypro4.2},
\ref{yypro4.3}, \ref{yypro4.5},  \ref{yypro4.7}, \ref{yypro4.9}, 
and \ref{yypro4.15}.
\end{proof}


\section{Quantum Kleinian singularities as noncommutative hypersurfaces} 
\label{yysec5}

Recall that if a semisimple Hopf algebra $H$ acts on an Artin-Schelter (AS) regular algebra $A$ inner faithfully, preserving the grading of $A$, with trivial homological determinant, and if further $A^H \neq A$, then $A^H$ is not AS regular \cite[Theorem~0.6]{CKWZ}. 
This motivates the following definition. 

\begin{definition} 
\label{yydef5.1}
We say that a ring $R$  has {\em{quantum Kleinian singularities}} if it is isomorphic to a fixed ring $A^{H}$ where:
\begin{enumerate}
\item[(i)] 
$H$ is a finite dimensional Hopf algebra, and
\item[(ii)] 
$A$ is an AS regular algebra of dimension~2 which is also a graded
inner-faithful $H$-module algebra, and 
\item[(iii)] the $H$-action on $A$ has 
trivial homological determinant.
\end{enumerate} 
\end{definition}

The classification of the quantum Kleinian  singularities follows
from Theorem~\ref{yythm0.4}. Here we describe these
fixed subrings as noncommutative hypersurface rings in the 
sense of \cite{KKZ4}. 

\begin{theorem}
\label{yythm5.2}
Let $H$ be a semisimple Hopf algebra and $A$ be an $H$-module algebra
satisfying Hypothesis~{\rm{\ref{yyhyp0.1}}} with $\gldim A =2$. Then $A^H$ 
is isomorphic to $C/C\Omega$, where $C$ is an AS-regular
algebra of dimension $3$ and $\Omega$ is a normal element in $C$.
\end{theorem}

\begin{proof}
The proof of this theorem consists of a case-by-case analysis, the results of which are summarized in Table~3. We defer the verification that $\Omega$ is normal in $C$ to the end of this proof.

In case ({\tt a}), the fixed subring $\kk [ u,v ]^{G}$, where 
$G \leq SL_{2}\left( \kk \right)$, is a (commutative) 
Kleinian singularity, and is well known to be a hypersurface 
singularity in affine 3-space. The cases of group actions 
(cases ({\tt b}), ({\tt c}), ({\tt d}), ({\tt g}), ({\tt h})) on noncommutative AS regular 
algebras holds by \cite[Theorem~0.1]{KKZ2} and \cite[Theorem 0.1]{KKZ5}. 
It remains to verify the result in cases ({\tt e}) and~({\tt f}). (Again, we show that $\Omega$ is normal in $C$ at the end.)

Case ({\tt e}): We show that $A^H = A^{\mathrm{co} \; \kk D_{2n}}$ as follows.   Let $v_{1} =u+ \sqrt{-1} v$ and $v_{2} = \sqrt{-1} u+v$. 
Then 
$$A:=\kk_{-1}[u,v]\cong\kk\langle v_{1},v_{2}\rangle/(v_{1}^{2}+v_{2}^{2}).$$ 
We use the following presentation for 
$$D_{2n} = \langle g , h \mid g^{2} =
h^{2} =( g h )^{n} =1 \rangle.$$  
The $\kk D_{2n}$-coaction on $A$ is given by 
$\rho(v_{1}) =v_{1}\otimes g$ and $\rho(v_{2}) =v_{2}\otimes h$.

\smallbreak Let $s_{r,k} := v^{r-2k}_{1} ( v_{1} v_{2} )^{k}$ and $s_{r,k}' :=
v^{r-2k}_{1} ( v_{2} v_{1} )^{k}$. We choose the following basis 
$\{ s_{2m,k} ,s'_{2m,k} \mid k=1, \ldots ,m \} 
\cup \{ v_{1}^{2m} \}$ for $A_{2m}$ and
the following basis 
$\{ s_{2m+1,k} ,s'_{2m+1,k} \mid k=1, \ldots ,m \} 
\cup \{v_{1} ( v_{1} v_{2} )^{2m} ,v_{2} ( v_{1} v_{2} )^{2m} \}$ 
for $A_{2m+1}$. It is immediate that $A^H = A^{\mathrm{co} \; \kk D_{2n}}$. 

\smallbreak Next we compute the $\kk D_{2n}$-coinvariants in 
degree $2m$. Clearly $v_{1}^{2m}$ is $\kk D_{2n}$-coinvariant, and 
$\rho ( s_{2m,k} ) =s_{2m,k} \otimes ( g h)^{k}$ implies 
that $s_{2m,k}$ is $\kk D_{2n}$-coinvariant if and only if $n \mid k$. 
Similarly $s'_{2m,k}$ is coinvariant if and only if $n \mid k$. 
We obtain a $\Bbbk$-basis of $\kk D_{2n}$-coinvariant elements
$\{s_{2m, p n}\}_{p\geq 0} \cup \{s'_{2m, pn}\}_{p\geq 0}.$
This argument shows that $$a_{1}=(v_{1} v_{2})^{n}, \quad
a_{2} =v_{1}^{2} =-v_{2}^{2}, \quad a_{3} = ( v_{2} v_{1} )^{n}$$ 
are $\kk D_{2n}$-coinvariant. These elements commute and satisfy $\Omega:=a_{2}^{2n} - ( -1 )^{n} a_{1} a_{3}$ equal to 0. So letting $C = \kk[a_1,a_2,a_3]$,
 we have a surjective ring homomorphism
\begin{eqnarray*}
\theta: C / ( \Omega ) &
\twoheadrightarrow & A^{\mathrm{co} \; \kk D_{2n}}.
\end{eqnarray*}

The $\Bbbk$-basis for $\Bbbk D_{2n}$-coinvariants given above shows that
\[
\begin{array}{rl}
H_{A^{\mathrm{co} \; \kk D_{2n}}} ( t ) & =  \displaystyle \sum_{j=0}^{\infty} \left( 1+2
\lfloor \textstyle \frac{j}{n} \rfloor \right) t^{2j} \quad =  (1+t^{2} +\cdots +t^{2n-2})\displaystyle\sum_{j=0}^{\infty} ( 1+2j ) t^{2j n}\\
\medskip

& =  \displaystyle\frac{(1+t^{2n})(1+t^{2} + \cdots +t^{2n-2})}{( 1-t^{2n} )^{2}}
\quad = \displaystyle \frac{1+t^{2} + \cdots +t^{4n-2}}{( 1-t^{2n} )^{2}}.
\end{array}
\]

The Hilbert series of $C/ ( \Omega )$ follows from noting that the
generators $a_1$, $a_2$, $a_3$ are in degrees $2n$, $2$, $2n$, respectively, with a relation $\Omega$ in degree $4n$.
Hence
\begin{eqnarray*}
H_{C/ ( \Omega )} ( t ) & = & \frac{1-t^{4n}}{( 1-t^{2n} )^{2} ( 1-t^{2} )}
.
\end{eqnarray*}
Since $H_{A^H}(t) = H_{A^{\mathrm{co} \;D_{2n}}}(t) =H_{C/ ( \Omega )}(t)$, we obtain that
$\theta$ is an isomorphism, as desired.

\smallskip

Case ({\tt f}): This case follows from Lemmas \ref{yylem3.1}(4), \ref{yylem3.2}(4), 
\ref{yylem3.3}(4), and \ref{yylem3.6}(4), and completes the proof. 

\smallskip

We summarize the cases in the following table.

{\small
\[ \begin{array}{|l|l|l|} \hline  &
 R = C/(\Omega) & (a_1,a_2,a_3) \\ 
 \hline 
 \hline 
{\rm{({\tt a})}} & \text{Commutative Kleinian singularity} & \\
\hline
{\rm{({\tt b})}} & C=\kk_{\mathbf{q}}[a_1,a_2,a_3] &  (u^n, uv, v^n) \\
   & \quad \text{where }q_{12}=q_{13}=(-1)^{n^2},q_{23}=(-1)^n & \\
    & \Omega =  a_2^n-(-1)^{\frac{n(n-1)}{2}}a_1 a_3  &                 \\ 
\hline
{\rm{({\tt c})}} & C = \kk\langle a_1,a_2\rangle/ ([a_1^2,a_2],[a_2^2,a_1]) 
                                           & (u+v, (u-v)(u+v)(u-v), 0) \\
    &\Omega = a_1^6-a_2^2   &  {\rm{(}}{\text{two generators}}{\rm{)}}\\
\hline
{\rm{({\tt d})}} & n\text{ even} & \\
    & C = \kk[a_1,a_2,a_3]                       & (u^n + v^n, (uv)^2, u^{n+1}v - uv^{n+1}) \\
    & \Omega =  
    a_3a_1+a_1a_3-4(-1)^{(n+2)/2}a_2^{(n+2)/2}
    & 
    \\
    \cline{2-3}
    & n\text{ odd}                                & \\
    & C=  \displaystyle\frac{\kk\langle a_1, a_2,a_3\rangle}
    {\left(\begin{array}{c}
    a_3a_1+a_1a_3-4(-1)^{(n+1)/2}a_2^{(n+1)/2} 
    \\
    a_2 \text{ central }
    \end{array}\right)}
		& (u^n + v^n, (uv)^2,u^{n+1}v + uv^{n+1})\\
    &\Omega= a_3^2 + a_2a_1^2 & \\
\hline
{\rm{({\tt e})}} & C= \kk[a_1,a_2,a_3] & ((v_1v_2)^n, v_1^2,(v_2v_1)^n) \\ 
    & \Omega = a_2^{2n}-(-1)^n a_1a_3              & \\
\hline 
{\rm{({\tt f})}} & \text{Commutative Kleinian singularity} & \\
\hline
{\rm{({\tt g})}} & C=\kk_{\mathbf{q}}[a_1,a_2,a_3]               & (u^n,uv, v^n)\\
    &\quad \text{where }q_{12}=q_{13}=q^{n^2},q_{23}=q^{n} & \\
    & \Omega =  a_2^n-q^{\frac{n(n-1)}{2}}a_1 a_3    & \\ 
\hline 
{\rm{({\tt h})}} & C= \displaystyle\frac{ \kk\langle a_1,a_2,a_3 \rangle}{
\left(\begin{array}{l}a_2a_1 -
a_1a_2 - 2a_1^2\\ a_3a_2 - a_2a_3 - 2a_2^2\\
a_3a_1 -a_1a_3 - 4a_1a_2 - 6a_1^2\end{array}\right)} & (u^2,uv,v^2) \\
    & \Omega = a_2^2-a_1a_2-a_1a_3  & \\ 
\hline 
\end{array} \]}
\medskip
\begin{center}
\small{\textsc{Table~3.} {\rm Quantum Kleinian singularities, as
noncommutative hypersurface singularities }}
\end{center}

Finally, in the cases when $C$ is noncommutative,  the normality of $\Omega$ in $C$ holds due to the computations below. The reader may wish to refer to Table~1.
{\small\[
\begin{array}{llllll}
({\tt b}) & a_1 \Omega = (-1)^{-n^2} \Omega a_1&&a_2 \Omega = \Omega a_2 &&a_3  \Omega = (-1)^{n^2} \Omega a_3 \\
({\tt c})  &a_1 \Omega = \Omega a_1 &&a_2 \Omega = \Omega a_2 &&\\
({\tt d}, n\text{ odd})  & a_1 \Omega = \Omega a_1 &&a_2 \Omega = \Omega a_2 &&a_3 \Omega = \Omega a_3  \\
({\tt g})  & a_1 \Omega = q^{-n^2}  \Omega a_1&&a_2 \Omega = \Omega a_2 && a_3 \Omega = q^{n^2}\Omega a_3 \\
({\tt h}) &a_1 \Omega = \Omega a_1  && a_2 \Omega = \Omega (2a_1+a_2)&& a_3\Omega = \Omega (6a_1+4a_2+a_3). 
\end{array}
\]}
\end{proof}

\subsection*{Acknowledgments}
The authors would like to thank W. Frank Moore for supplying corrections to some computations presented in Table~3, and we thank the referee for providing several helpful suggestions.
C. Walton and J.J. Zhang were supported by the US National Science
Foundation: NSF grants DMS-1550306, 1663775 and DMS-1402863
respectively. C. Walton is also supported by a research fellowship from the Sloan Foundation. E. Kirkman was supported by Simons Grant \#208314.

\bibliography{hopfmckay}

\end{document}